\documentclass[11pt,a4paper]{amsart}
\usepackage{amssymb,xspace}
\usepackage{amstext}
\usepackage{tikz}
\usetikzlibrary{arrows,decorations.pathmorphing,backgrounds,fit,positioning,shapes.symbols,chains}
\theoremstyle{plain}
\usepackage{amsbsy,amssymb,amsfonts,latexsym}

\usepackage{enumerate}
\usepackage{graphics}

\date{\today}
\title[General Dirichlet series]{Convergence and almost sure properties in Hardy spaces of Dirichlet series}

\author{Fr\'{e}d\'{e}ric Bayart}
\address{Laboratoire de Mathématiques Blaise Pascal UMR 6620 CNRS, Université Clermont Auvergne, Campus universitaire des Cézeaux, 3 place Vasarely, 63178 Aubière Cedex, France.} 
\email{frederic.bayart@uca.fr}
\thanks{The author was partially supported by the grant ANR-17-CE40-0021 of the French National Research Agency ANR (project Front).}



\newcommand{\veps}{\varepsilon}

\def\QQ{\mathbb Q}

\def\RR{\mathbb R}

\def\NN{\mathbb N}
\def\ZZ{\mathbb Z}
\def\TT{\mathbb T}

\def\CC{\mathbb C}

\def\dinfext{\mathcal D_\infty^{\textrm{ext}}(\lambda)}
\def\dinf{\mathcal D_\infty(\lambda)}
\def\pollambda{\textrm{Pol}_\lambda(G)}

\DeclareMathOperator{\card}{card}

\newtheorem{theorem}{Theorem}[section]

\newtheorem{lemma}[theorem]{Lemma}

\newtheorem{proposition}[theorem]{Proposition}

\newtheorem{corollary}[theorem]{Corollary}

{\theoremstyle{definition}}
{\theoremstyle{definition}}

{\theoremstyle{definition}\newtheorem{example}[theorem]{Example}}

{\theoremstyle{definition}\newtheorem{definition}[theorem]{Definition}}

{\theoremstyle{definition}}

{\theoremstyle{definition}\newtheorem{remark}[theorem]{Remark}}

\newtheorem{question}[theorem]{Question}

\newtheorem*{theoremDS}{Theorem A}

\begin{document}

\begin{abstract}
Given a frequency $\lambda$, we study general Dirichlet series $\sum a_n e^{-\lambda_n s}$. First, we give a new condition on $\lambda$ which ensures that a somewhere convergent Dirichlet series defining a bounded holomorphic function in the right half-plane converges uniformly in this half-plane, improving classical results of Bohr and Landau. Then, following recent works of Defant and Schoolmann, we investigate Hardy spaces of these Dirichlet series. We get general results on almost sure convergence which have an harmonic analysis flavour. Nevertheless, we also exhibit examples showing that it seems hard to get general results on these spaces as spaces of holomorphic functions.
\end{abstract}

\maketitle

\section{Introduction}
\subsection{$\lambda$-Dirichlet series and their convergence}
A general Dirichlet series is a series $\sum_n a_n e^{-\lambda_n s}$ where $(a_n)\subset\CC^\mathbb N$, $s\in\CC$ and $\lambda=(\lambda_n)$ is an increasing sequence of nonnegative real numbers tending to $+\infty$, called a \emph{frequency}. We shall denote by $\mathcal D(\lambda)$ the space of all formal $\lambda$-Dirichlet series. The two most natural examples are $(\lambda_n)=(n)$ which gives rise to power series and $(\lambda_n)=(\log n)$, the case of ordinary Dirichlet series. A proeminent problem at the beginning of the twentieth century was the study of the convergence of these series, starting from the following theorem of Bohr \cite{Bohr13} on ordinary Dirichlet series: let $D=\sum_n a_n n^{-s}$ be a somewhere convergent ordinary Dirichlet series having a holomorphic and bounded extension to the half-plane $\CC_0$. Then $D$ converges uniformly on each half-place $\CC_\veps$ for all $\veps>0$. Here, $\CC_\theta$ means the right half-plane $[\mathrm{Re}s>\theta]$. 

An important effort was done to extend this result to general frequencies. Two sufficient conditions were isolated firstly by Bohr \cite{Bohr13} and then by Landau \cite{Landau21}. Following the terminology of \cite{Schoolmann20}, we say that a frequency satisfies (BC) provided 
\begin{equation}\label{eq:bc}\tag{BC}
\exists \ell>0,\ \exists C>0,\ \forall n\in\NN,\ \lambda_{n+1}-\lambda_n\geq Ce^{-\ell\lambda_n}.
\end{equation}
A frequency $\lambda$ satisfies (LC) provided
\begin{equation}\label{eq:lc}\tag{LC}
\forall \delta>0,\ \exists C>0,\ \forall n\in\NN,\ \lambda_{n+1}-\lambda_n \geq Ce^{-e^{\delta\lambda_n}}.
\end{equation}
Of course, (BC) is a stronger condition than (LC), and Landau has shown that any frequency satisfying (LC) verifies Bohr's theorem: all Dirichlet series $D=\sum_n a_n e^{-\lambda_n s}$ belonging to $\dinfext$, the space of somewhere convergent $\lambda$-Dirichlet series which allow a holomorphic and bounded extension to $\CC_0$, converge uniformly to $\CC_\veps$, for all $\veps>0$.

This problem was studied again recently in \cite{CKM18} and \cite{Schoolmann20} in connection with Banach spaces of Dirichlet series and quantitative estimates. Define by $\dinf$ the subspace of $\dinfext$ of Dirichet series which converge on $\CC_0$ (in general, $\dinf$ can be a proper subspace of $\dinfext$ if $\lambda$ does not satisfy Bohr's theorem) and define $S_N:\dinfext\to\dinf$, $\sum_{1}^{+\infty}a_n e^{-\lambda_n s}\mapsto \sum_1^N a_n e^{-\lambda_n s}$ the $N$-th partial sum operator. In \cite{Schoolmann20}, a thorough study of the norm of $S_N$ is done (the case of ordinary Dirichlet series was settled in \cite{Baymonat}), leading to consequences on the existence of a Montel's theorem in $\dinfext$ or on the completeness of $\dinf$.

Our first main result is an extension of the results of \cite{Bohr13,Landau21,Schoolmann20}: we provide a new class of frequencies, that we will call (NC) (see Definition \ref{def:NC}) such that Bohr's theorem, and all its consequences, are true. Like (BC) or (LC), this class of frequencies quantifies how fast $\lambda$ goes to $+\infty$ and how close its terms are, but in a less demanding way since (NC) is strictly weaker than (LC).

Our method of proof also differs from that of \cite{Schoolmann20}. In \cite{Schoolmann20}, the estimation of $\|S_N\|$ is based on Riesz means of $\lambda$-Dirichlet polynomials: recall that for a sequence of complex numbers $(c_n)$ and for $k>0$, the finite sum
$$R_x^{\lambda,k}(\sum_n c_n):=\sum_{\lambda_n<x}\left(1-\frac{\lambda_n}x\right)^k c_n$$
is called the first $(\lambda,k)$-Riesz mean of $\sum_n c_n$ of length $x>0$. I. Schoolmann uses approximation of $D$ by $R_{x}^{\lambda,k}(D)$ for a suitable choice of $k$ to deduce its result on $\|S_N\|$. Our alternate approach is based on mollifiers and on a formula of convolution due to Saksman for ordinary Dirichlet series.

\subsection{Hardy spaces of $\lambda$-Dirichlet series: Banach spaces and harmonic analysis}

Our second approach deals with Hardy spaces of Dirichlet series. For ordinary Dirichlet series, they have been introduced and studied in \cite{HLS} (see also \cite{Baymonat}) and this has caused an important renew of interest for this subject. The general case has been introduced and investigated very recently in \cite{DS19,DSHelson,DSRiesz}. For $p\in[1,+\infty)$, the Hardy space $\mathcal H_p(\lambda)$ may be defined as follows: given a $\lambda$-Dirichlet polynomial $D=\sum_{n=1}^N a_n e^{-\lambda_n s}$, define its $\mathcal H_p$-norm by
$$\|D\|_p^p:=\lim_{T\to+\infty}\frac 1{2T}\int_{-T}^T |D(it)|^p dt.$$
Then $\mathcal H_p(\lambda)$ is the completion of the set of $\lambda$-Dirichlet polynomials for this norm. However, this internal description is often not sufficient to get the main properties of $\mathcal H_p(\lambda)$ and we need a group approach. Let $G$ be a compact abelian group and let $\beta:(\RR,+)\to G$ be a continuous homomorphism with dense range. Then we say that $(G,\beta)$ is a $\lambda$-Dirichlet group provided, for all $n\in\NN$, there exists $h_{\lambda_n}\in\hat G$ such that $h_{\lambda_n}\circ\beta=e^{-i\lambda_n\cdot}$. The space $H_p^\lambda(G)$, $p\in[1,+\infty]$ is then defined as the subspace of $L^p(G)$ of functions $f$ such that $\textrm{supp}(\hat f)\subset \{h_{\lambda_n}:\ n\in\NN\}$. Now define the Bohr map $\mathcal B$ by 
\begin{align*}
\mathcal B:H_p^{\lambda}(G)&\to\mathcal D(\lambda)\\
f&\mapsto \sum \hat{f}(h_{\lambda_n})e^{-\lambda_n s}
\end{align*}
and set $\mathcal H_p(\lambda)=\mathcal B(H_p^{\lambda}(G))$ with $\|\mathcal Bf\|_p:=\|f\|_p$. Then it has been shown in \cite{DS19} that
\begin{itemize}
\item given a frequency $\lambda$, there always exists a $\lambda$-Dirichlet group $(G,\beta)$;
\item the Hardy space $\mathcal H_p(\lambda)$ does not depend on the choosen $\lambda$-Dirichlet group;
\item when $p\neq+\infty$, it coincides with $\mathcal H_p(\lambda)$ defined internally.
\end{itemize}
Our second aim in this paper is solve some of the problems on the spaces $\mathcal H_p(\lambda)$ raised in \cite{DSsurvey} and to exhibit new properties of them. In particular, we investigate properties of $\mathcal H_p(\lambda)$ coming from functional analysis and harmonic analysis.

\smallskip

As an example, let us discuss a famous therorem of Helson \cite{Hel69} which ensures, in our terminology, that if $\lambda$ satisfies (BC) and $D=\sum_n a_n e^{-\lambda_ns}$ belongs to $\mathcal H_2(\lambda)$, then for almost all homomorphisms $\omega:(\RR,+)\to \TT$, the Dirichlet series $\sum_n a_n \omega(\lambda_n) e^{-\lambda_n s}$ converges on $\CC_0$. This has been extended to the Hardy spaces $\mathcal H_p((\log n))$ for $p\geq 1$ by Bayart in \cite{Baymonat} and this result is at the heart of many further investigations on these spaces (e.g. composition operators, Volterra operators). Therefore, it is a challenge to put it in the general framework of  $\mathcal H_p(\lambda)$ or, equivalenty - via the Bohr transform - of $H_p^\lambda(G)$. When $\lambda$ satisfies (BC), this has been done in \cite{DSHelson}, adding moreover the maximal inequality.
\begin{theoremDS}[Defant-Schoolmann]
Let $\lambda$ satisfy (BC), let $(G,\beta)$ be a $\lambda$-Dirichlet group. For every $u>0$, there exists a constant $C:=C(u,\lambda)$ such that, for all $1\leq p\leq +\infty$ and for all $f\in H_p^{\lambda}(G)$, 
\begin{equation}\label{eq:DSHelson}
\left\|\sup_{\sigma\geq u}\sup_N \left|\sum_{n=1}^N \hat{f}(h_{\lambda_n})e^{-\sigma \lambda_n}h_{\lambda_n}\right|\right\|_p\leq C\|f\|_p.
\end{equation}
In particular, for every $u>0$, $\sum_{1}^{+\infty}\hat f(h_{\lambda_n}) e^{-u \lambda_n}h_{\lambda_n}$ converges almost everywhere on $G$. 
\end{theoremDS}
When $\lambda$ satisfies (LC), the almost everywhere statement is known to be true, as well as the maximal inequality for $p>1$ with a constant now depending on $p$. When $p=1$, it is valid if we replace the $L_1(G)$-norm by the weak $L_1(G)$-norm. We shall prove that inequality \eqref{eq:DSHelson} remains true even on the weaker assumption that $\lambda$ satisfies (NC), even for $p=1$, and with a constant independent of $p$. Our approach, which seems less technical than that of \cite{DSHelson}, is based again on a version of Saksman's convolution formula and on a Carleson-Hunt type maximal inequality of independent interest.

\subsection{$\mathcal H_p(\lambda)$ as a Banach space of holomorphic functions}
The results announced in the previous section indicate that the spaces $\mathcal H_p(\lambda)$ seem to behave well if we look at their almost sure properties. The classical case $\mathcal H_p((\log n))$ was also investigated as a Banach space of holomorphic function. Even in that case, it is a nontrivial problem to determine the optimal 
half-plane of convergence of elements in $\mathcal H_p(\lambda)$, namely to compute
$$\sigma_{\mathcal H_p(\lambda)}:=\inf\{\sigma\in\RR:\ \sigma_c(D)\leq \sigma\textrm{ for all }D\in\mathcal H_p(\lambda)\}$$
where, for a Dirichlet series $D\in\mathcal D(\lambda)$, $\sigma_c(D):=\inf\{\sigma\in\RR:\ D\textrm{ converges on } \CC_\sigma\}$. This has been settled in \cite{Baymonat}, using that $\sigma_{\mathcal H_2((\log n))}=1/2$ (easy by the Cauchy-Schwarz inequality) and that $T_\sigma(\sum_n a_n e^{-\lambda s})=\sum_n a_n e^{-\sigma \lambda_n}e^{-\lambda_n s}$ 
maps $\mathcal H_p((\log n))$ into $\mathcal H_q((\log n))$ for all $p,q\in[1,+\infty)$ and all $\sigma>0$. 
The argument is based on multiplicativity (namely on the fact that  the natural $\lambda$-Dirichlet group for $(\log n)$ is the infinite polytorus $\TT^\infty$) and on a hypercontractive estimate for the Poisson kernel acting on the Hardy spaces $H^p(\TT)$ of the disk.

We will show that there is no hope to get such a result for general frequencies $\lambda$ even if they satisfy (BC). For instance, if we will be able to prove that for all frequencies $\sigma_{\mathcal H_1(\lambda)}\leq 2\sigma_{\mathcal H_2(\lambda)}$, we will nevertheless point out that, even if we assume (BC), this is optimal and in particulat that we may have $\sigma_{\mathcal H_1(\lambda)}>\sigma_{\mathcal H_2(\lambda)}$. We will also exhibit  a sequence $\lambda$, which still satisfies (BC), such that $T_\sigma$ maps boundedly $\mathcal H_2(\lambda)$ into $\mathcal H_{2k}(\lambda)$ if and only if $\sigma\geq (k-1)/2k$. In particular, it seems very hard to compute $\sigma_{\mathcal H_p(\lambda)}$ in the general case and the behaviour of $\mathcal H_p(\lambda)$ as a space of holomorphic function seems more difficult to predict if we assume only growth and separation conditions on $\lambda$.

\subsection{Notations}
Throughout this work, we shall use the following notations. For $\lambda$ a frequency and $D\in\mathcal D(\lambda)$, the abscissa of absolute convergence of $D$ and the abscissa of uniform convergence of $D$ are defined by 
\begin{align*}
\sigma_a(D)&:=\inf\{\sigma\in\RR:\ D\textrm{ converges absolutely on }\CC_\sigma\}\\
\sigma_u(D)&:=\inf\{\sigma\in\RR:\ D\textrm{ converges uniformly on }\CC_\sigma\}.
\end{align*}
We set 
\begin{align*}
L(\lambda)&:=\limsup_{N\to+\infty}\frac{\log N}{\lambda_N}\\
&=\sup_{D\in\mathcal D(\lambda)} \sigma_a(D)-\sigma_c(D).
\end{align*}

Given $(G,\beta)$ a $\lambda$-Dirichlet group we shall denote by $\textrm{Pol}_\lambda(G)$ the set of polynomials with spectrum in $\lambda$, namely finite sums $\sum_{k=1}^n a_k h_{\lambda_k}$ with $\lambda_k\in G$ for each $k=1,\dots,n$. We shall also use the following result: for all $f:G\to\CC$ measurable, for almost all $\omega\in G$, the function $f_\omega:=f(\omega\beta(\cdot)):\RR\to\CC$ is measurable. If additionally $f\in L_\infty(G)$, then for almost all $\omega\in G$, $f_\omega\in L_\infty(\RR)$ with $\|f_\omega\|_\infty\leq \|f\|_\infty$. Moreover, if $f\in L_1(G)$, then $f_\omega$ is locally integrable for almost all $\omega\in G$, and for $g\in L_1(\RR)$, the convolution 
$$g\star f_\omega(t):=\int_{\RR} f(\omega\beta(t-y))g(y)dy$$
is almost everywhere defined on $\RR$ and measurable (see \cite[Lemma 3.11]{DS19}).

\section{Preliminaries}
\subsection{A new class of frequencies}\label{sec:nc}

We introduce our new condition, more general than (LC), under which most of our results will be satisfied. We first reformulate (LC).

\begin{lemma}
A frequency $\lambda$ satisfies (LC) if and only if there exists $C>0$ such that, for all $\delta>0$, for all $n\in\NN$, 
$$\log\left(\frac{\lambda_{n+1}+\lambda_n}{\lambda_{n+1}-\lambda_n}\right)\leq Ce^{\delta\lambda_n}.$$
\end{lemma}
\begin{proof}
Assume first that $\lambda$ satisfies (LC) and let $\delta>0$. Then there exists $C>0$ such that, for all $n\in\NN$, $\lambda_{n+1}-\lambda_n \geq Ce^{-e^{{\frac\delta 2}\lambda_n}}$. Let $n\in\NN$ and set $\xi_n=\lambda_n+Ce^{-e^{\frac\delta2\lambda_n}}$. Since the function $x\mapsto (x+\lambda_n)/(x-\lambda_n)$ is decreasing on $(\lambda_n,+\infty)$, one gets
\begin{align*}
\frac{\lambda_{n+1}+\lambda_n}{\lambda_{n+1}-\lambda_n}&\leq \frac{\xi_n+\lambda_n}{\xi_n-\lambda_n}\\
&\leq C^{-1}\left(2\lambda_n+Ce^{-e^{\frac\delta 2\lambda_n}}\right)e^{e^{\frac\delta 2\lambda_n}}\\
&\leq C' e^{e^{\delta\lambda_n}}.
\end{align*}
The converse implication is easier and left to the reader.
\end{proof}
The main idea to introduce (NC) is to allow to compare the position of $\lambda_n$ with $\lambda_m$ for some $m>n$ and not only with $\lambda_{n+1}$. 
\begin{definition}\label{def:NC}
We say that a frequency $\lambda$ satisfies (NC) if, for all $\delta>0$, there exists $C>0$ such that, for all $n\geq 1$, there exists $m>n$ such that
\begin{equation}\tag{NC}
\log\left(\frac{\lambda_{m}+\lambda_n}{\lambda_{m}-\lambda_n}\right)+(m-n)\leq Ce^{\delta\lambda_n}.
\end{equation}
\end{definition}
Condition (NC) provides a nontrivial extension of (LC).
\begin{example}\label{ex:sequenceNC}
Let $\lambda$ be defined by $\lambda_{2^n+k}=n^2+ke^{-e^{n^2}}$ for $k=0,\dots,2^n-1$. Then $L(\lambda)=+\infty$, $\lambda$ satisfies (NC) and $\lambda$ is not the finite union of frequencies satisfying (LC).
\end{example}
\begin{proof}
Let $\delta>0$, $n\in\mathbb N$, $k\in\{0,\dots,2^n-1\}$, then provided $n$ is large enough
\begin{align*}
\log\left(\frac{\lambda_{2^{n+1}}+\lambda_{2^n+k}}{\lambda_{2^{n+1}}-\lambda_{2^n+k}}\right)+(2^{n+1}-2^n-k)&\leq \log\big(2(n+1)^2\big)+2^n\\
&\leq C e^{\delta n^2}\\
&\leq Ce^{\delta\lambda_{2^n+k}}
\end{align*}
for some $C>0$. Moreover, if $\lambda$ was the finite union of $\lambda^1,\dots,\lambda^p$, each $\lambda^j$ satisfying (LC), then at least one of the $\lambda^j$, say $\lambda^1$, will contain an infinite number of consecutive terms $\lambda_m^1=\lambda_{2^n+k}$, $\lambda_{m+1}^1=\lambda_{2^n+k'}$ with $1\leq k'-k\leq p$ and $k,k'\in \{0,\dots,2^n-1\}$. For these $m$,
$$\log\left(\frac{\lambda_{m+1}^1+\lambda_m^1}{\lambda_{m+1}^1-\lambda_m^1}\right)\geq e^{n^2}-\log p\geq Ce^{\lambda_m^1/2}$$
contradicting that $\lambda^1$ satisfies (LC).
\end{proof}

\subsection{Saksman's vertical convolution formula}
Saksman's vertical convolution formula was introduced to express weighted sums of ordinary Dirichlet series using an integral. It says essentially that if $D=\sum_n a_n n^{-s}$ is an ordinary Dirichlet series and $\psi$ is in $L^1$ with $\hat \psi$ compactly supported, then 
$$\sum_{n=1}^{+\infty}a_n \hat\psi(\log n)n^{-s}=\int_{\RR}D(s+it)\psi(t)dt$$
with a sense that has to be made precise. It was used in \cite{BQS16} for Dirichlet series in $\mathcal H_1$ and in \cite{Qu15} for Dirichlet series in $\mathcal H_\infty$. We shall extend it to general Dirichlet series and we will use it as a much more flexible substitute of Perron's formula.

\begin{theorem}\label{thm:saksman}
Let $\psi\in L^1(\RR)$ be such that $\hat\psi$ is compactly supported and let $\lambda$ be a frequency with $(G,\beta)$ an associated $\lambda$-Dirichlet group.
\begin{enumerate}[(a)]
\item Let $D=\sum_n a_n e^{-\lambda_n s}\in\dinfext$ with bounded and holomorphic extension to $\CC_0$ denoted by $f$. Then for all $s\in\CC$ with $\Re e (s)>0$
$$\sum_{n=1}^{+\infty}a_n\hat\psi(\lambda_n)e^{-\lambda_n s}=\int_{\RR}f(s+it)\psi(t)dt.$$
\item Let $f=\sum_n a_n h_{\lambda_n}\in H_1^\lambda(G)$. Then for almost all $\omega\in G$, 
$$\sum_{n=1}^{+\infty}a_n \hat\psi(\lambda_n)h_{\lambda_n}(\omega)=\int_{\RR}f_\omega(t)\psi(t)dt.$$
\end{enumerate}
\end{theorem}

In the sequel, for $D=\sum_n a_n e^{-\lambda_n s}\in\mathcal \dinfext$, respectively for $f=\sum_n a_n h_{\lambda_n}$ in $H_1^\lambda(G)$, and for $\psi\in L^1(\RR)$ compactly supported, we shall denote
\begin{align*}
R_\psi(D)&:=\sum_n a_n \hat{\psi}(\lambda_n)e^{-\lambda_n s}\\
R_\psi(f)&:=\sum_n a_n \hat{\psi}(\lambda_n) h_{\lambda_n}.
\end{align*}

\begin{proof}
(a) Observe first that the equality is true provided $D$ is a Dirichlet polynomial and that the two members of the equality define an analytic function on $\CC_0$. Assume first that $L(\lambda)<+\infty$. Then $\sigma_a(D)<+\infty$ and for $s>\sigma_a(D)$, the formula is true just by exchanging the sum and the integral. We conclude by analytic continuation. 

When $L(\lambda)=+\infty$, the proof is more difficult. We use (see \cite[Theorem 41 p. 53]{HarRi} or \cite[Theorem 22]{Hel05}) that there exist a half-plane $\CC_\theta$, $\theta>0$, and a sequence of $\lambda$-Dirichlet polynomials $D_j=\sum_{n=1}^{+\infty} a_n^j e^{-\lambda_n s}$ such that $(a_n^j)$ tends to $a_n$ as $j$ tends to $+\infty$ for any $n$ and $(D_j)$ converges uniformly to $f$ on $\CC_\theta$. Since each $D_j$ is a Dirichlet polynomial, we know that for all $s\in\CC_\theta$ and all $j\in\NN$, 
$$\sum_{n=1}^{+\infty}a_n^j \hat\psi(\lambda_n)e^{-\lambda_n s}=\int_{\RR}D_j(s+it)\psi(t)dt.$$
Letting $j$ to $+\infty$ in the previous inequality for a fixed $s\in\CC_\theta$, since the sum on the left handside is finite ($\hat\psi$ has compact support), and by uniform convergence, we get the result on $\CC_\theta$. We conclude again by
analytic continuation.\\
(b) When $f\in \pollambda$, the equality follows immediately by interverting a finite sum and an integral, and the definition of the objects that come into play:
\begin{align*}
R_\psi(f)(\omega)&=\int_{\RR}\sum_n   a_n e^{-it\lambda_n}\psi(t)h_{\lambda_n}(\omega) dt\\
&=\int_{\RR}\sum_n a_n h_{\lambda_n}(\beta(t))h_{\lambda_n}(\omega)\psi(t)dt\\
&=\int_{\RR}f_\omega(t)\psi(t)dt
\end{align*}
(here the equality is valid for all $\omega\in G$).
Let now $f\in H_1^\lambda(G)$. Then 
$$\int_G\int_\RR |f_\omega(t)\psi(t)|dt \leq \|f\|_1\|\psi\|_1.$$
Therefore, for almost all $\omega\in G$, the function $t\mapsto f_\omega(t)\psi(t)$ belongs to $L_1(\RR)$
and the operator $S_\psi:H_1^\lambda(G)\to L_1(G,L_1(\RR))$, $f\mapsto [\omega\mapsto f_\omega(\cdot)\psi(\cdot)]$ is continuous. If $(f_n)$ is a sequence in $\pollambda$ tending to $f\in H_1^{\lambda}(G)$, then there exists a sequence $(n_k)$ such that, for a.e. $\omega\in G$, 
\begin{align*}
S_{\psi}(f_{n_k})(\omega)&\to S_\psi(f)(\omega)\textrm{ in }L_1(\RR)\\
R_\psi(f_{n_k})(\omega)&\to R_\psi(f)(\omega)
\end{align*}
(recall that the sum defining $R_\psi$ is finite). Since $R_\psi(f_{n_k})(\omega)=\int_\RR S_\psi(f_{n_k})(\omega)$ for all $k$ and all $\omega\in G$, we get the conclusion by taking the limit.
\end{proof}

\begin{remark}
The statement of Theorem \ref{thm:saksman} remains true provided $\psi$ is not compactly supported but still satisfies $\sum_n |\hat\psi(\lambda_n)|<+\infty$.
\end{remark}
\begin{remark}
To obtain Theorem \ref{thm:saksman}, in both cases, we use the density of polynomials for a suitable topology. In $H_1^\lambda(G)$, this is trivial which is not the case in $\mathcal D_{\rm ext}^{\infty}(\lambda)$. More specifically
we intend to use Theorem \ref{thm:saksman} to obtain results that do not seem easily reachable using Riesz means. Therefore
it is intesting to obtain a proof of Theorem \ref{thm:saksman} that do not use Riesz means. This is the case if we use \cite[Theorem 22]{Hel05}. We thank A. Defant and I. Schoolmann for pointing out to me this reference.
\end{remark}
\begin{remark}
Part (b) of the vertical convolution formula is more precised than the statement established and used in \cite{BQS16}. The equivalent statement in this context would be that, for all $f\in H_1^\lambda(G)$, 
$$\sum_{n=1}^{+\infty}a_n \hat\psi(n)h_{\lambda_n}=\int_{\RR}T_t f \psi(t)dt,$$
where $T_t:H^\lambda_1(G)\to H^\lambda_1(G),\ f\mapsto f(\beta(t)\cdot)$ is an onto isometry of $H_1^\lambda(G)$ and the right handside denotes a vector-valued integral in $H_1^\lambda(G)$. We will need a pointwise statement in order to obtain maximal inequalities.
\end{remark}	
\subsection{Riesz means and Saksman's vertical convolution formula}
We now show how the results on $(\lambda,k)$-Riesz means will follow from our results coming from Saksman's vertical convolution formula. This is a consequence of the following easy proposition.
\begin{proposition}
Let $\alpha>0$. Then there exists an $L^1(\mathbb R)$-function $\psi$ such that, for all $t\in\RR$,
$\hat{\psi}(x)=(1-|x|)^\alpha$ provided $|x|<1$, $\hat{\psi}(x)=0$ otherwise.
\end{proposition}
\begin{proof}
Define $u(x)=(1-|x|)^\alpha \mathbf 1_{[-1,1]}(x)$. Then $u$ is piecewise $\mathcal C^1$, its derivative $u'(x)=\pm\alpha(1-|x|)^{\alpha-1}\mathbf 1_{[-1,1]}(x)$ belongs to $L^1(\mathbb R)$ and thus we know that, for all $t\neq 0$, 
$\hat u(t)=\frac{1}{it}\widehat{u'}(t).$
Now, it is easy to see that $u'$ belongs to $L^{1+\veps}(\RR)\cap L^1(\RR)$ for some $\veps>0$. Hence, $\widehat{u'}$ belongs to $L^q(\RR)$ for some $q<+\infty$. In particular, by H\"older's inequality, $\hat u$ belongs to $L^1$, so that we may apply the inverse Fourier transform to get the statement.
\end{proof}

In the sequel, for $\lambda$ a frequency, $\alpha>0$, $(G,\beta)$ a $\lambda$-Dirichlet group and $N>0$, we shall use the following notations:
\begin{align*}
R_N^{\lambda,\alpha}(f)&=\sum_{\lambda_n\leq N} \widehat f(h_{\lambda_n})\left(1-\frac{\lambda_n}N\right)^\alpha h_{\lambda_n}\\
R_N^{\lambda,\alpha}(D)&=\sum_{\lambda_n\leq N} a_n\left(1-\frac{\lambda_n}N\right)^\alpha e^{-\lambda_n s}
\end{align*}
where $f\in H_1^\lambda(G)$ and $D=\sum_n a_n e^{-\lambda_n s}\in\mathcal D(\lambda)$.  Many of the results of \cite{Schoolmann20,DSHelson,DSRiesz} are based on a detailed study of these operators $R_N^{\lambda,\alpha}$. We shall extend them via the convolution formula to other operators $R_\psi$, allowing better results with a different choice of $\psi$.


\section{Bohr's theorem under (NC)}

\subsection{The case of $\dinfext$}

In his study of Bohr's theorem \cite{Schoolmann20}, I. Schoolmann used that, for all $D=\sum_n a_ne^{-\lambda_n s}\in\dinfext$ with extension $f$, the sequence of its Riesz means of order $k$
$$R_x^k(D)=\sum_{\lambda_n<x} a_n\left(1-\frac{\lambda_n}x\right)^k e^{-\lambda_n s}$$
converge uniformly to $f$ on each halfplane $\CC_\veps$, for all $\veps>0$, as $x\to+\infty$. We now show that we may replace the function $\psi$ such that $\hat \psi(t)=(1-|t|)^k\mathbf 1_{[-1,1]}(t)$ by any function $L^1$-function $\psi$ such that $\hat\psi$ has compact support. 

\begin{lemma}\label{lem:Bohr1}
Let $\lambda$ be a frequency, let $\psi\in L^1(\RR)$ be such that $\hat \psi$ has compact support and let $D=\sum_n a_n e^{-\lambda_n s}\in\dinfext$ with extension $f$. Then 
$$\left\|R_{\psi}(D)\right\|_\infty\leq \|\psi\|_1 \|f\|_\infty.$$
Moreover, if $\int_{\RR}\psi=1$,denoting by $\psi_N(\cdot)=N\psi(N\cdot)$, the sequence of $\lambda$-Dirichlet polynomials 
$(R_{\psi_N}(D))$ converges uniformly to $f$ on each half-plane $\CC_\veps$, for all $\veps>0$. 
\end{lemma}

\begin{proof}
The inequality follows immediately from Theorem \ref{thm:saksman}. The statement on uniform convergence follows as well from this formula and from standard results on mollifiers, provided we know that $f$ is uniformly continuous on $\CC_\veps$. Again, this can be deduced from the fact that on this half-plane, $f$ is the uniform limit of $\lambda$-Dirichlet polynomials, which are themselves uniformly continuous.
\end{proof}  

We shall now apply this to a suitable choice of $\psi$ in order to get good estimates of the norm of the projection $S_N$.

\begin{theorem}\label{thm:BohrNC}
Let $\lambda$ be a frequency. There exists $C>0$ such that, for all $M>N\geq 1$, 
$$\|S_N\|_{\dinfext\to\dinf}\leq C\left(\log\left(\frac{\lambda_M+\lambda_N}{\lambda_M-\lambda_N}\right)+(M-N-1)\right).$$
\end{theorem}
\begin{proof}
We set $h=\frac{\lambda_M-\lambda_N}2$. Let $u$ be the function equal to $1$ on $[-\lambda_N,\lambda_N]$, to $0$ on $\RR\backslash(-\lambda_M,\lambda_M)$, and which is affine on $(-\lambda_M,\lambda_N)$ and
on $(\lambda_N,\lambda_M)$. The function $u$ may be written
$$u=\mathbf 1_{[-\lambda_N-h,\lambda_N+h]}\star\left(\frac1{2h} \mathbf 1_{[-h,h]}\right).$$
This formula allows us to compute the Fourier transform of $u$ which is equal to 
$$\widehat u(t)=2\frac{\sin((\lambda_N+h)t)}{t}\cdot\frac{\sin(ht)}{ht}$$
which is an $L^1$ function. Moreover
\begin{align*}
\|\widehat u\|_1&\leq 4\int_0^{+\infty}\left|\frac{\sin((\lambda_N+h)t)}{t}\right|\times \left|\frac{\sin(ht)}{ht}\right|dt\\
&\leq 4\int_0^{+\infty}\left| \frac{\sin\left(\frac{\lambda_N+h}h x\right)}{x}\right|\times \left|\frac{\sin(x)}{x}\right|dx\\
&\leq 4\int_0^1 \left| \frac{\sin\left(\frac{\lambda_N+h}h x\right)}{x}\right| dx+4\int_1^{+\infty}\frac1{x^2}dx\\
&\leq C\log\left(\frac{\lambda_N+h}{h}\right)+4=C\log\left(\frac{\lambda_M+\lambda_N}{\lambda_M-\lambda_N}\right)+4
\end{align*}
where we have used well-known estimates of the $L^1$-norm of the sinus cardinal function. We then applied Lemma \ref{lem:Bohr1} to $\psi\in L^1$ defined by $\widehat\psi=u$. By the Fourier inverse formula, 
$$\left\|\sum_{n=1}^{M-1}a_n\widehat{\psi}(\lambda_n)e^{-\lambda_n s}\right\|_\infty\leq 
\left(C\log\left(\frac{\lambda_M+\lambda_N}{\lambda_M-\lambda_N}\right)+4\right)\|f\|_\infty.$$
We get the conclusion by writing
\begin{eqnarray*}
\sum_{n=1}^N a_n e^{-\lambda_n s}=\sum_{n=1}^{M-1}a_n\widehat{\psi}(\lambda_n)e^{-\lambda_n s}-\sum_{n=N+1}^{M-1}a_n\widehat{\psi}(\lambda_n)e^{-\lambda_n s}
\end{eqnarray*}
and by using that $|a_n|\leq \|f\|_\infty$ (see \cite[Corollary 3.9]{Schoolmann20}) and  $\|\widehat{\psi}\|_\infty\leq 1$.
\end{proof}

From the Bohr-Cahen formula to compute the abscissa of uniform convergence of a $\lambda$-Dirichlet series,
$$\sigma_u(D)\leq \limsup_N \frac{\log\left(\sup_{t\in\RR}\left|\sum_{n=1}^N a_n e^{-\lambda_n it}\right|\right)}{\lambda_N}$$
we get the following corollary.
\begin{corollary}\label{cor:Bohr}
Let $\lambda$ be a frequency satisfying (NC). Then $\lambda$ satisfies Bohr's theorem.
\end{corollary}

Let us now compare Theorem \ref{thm:BohrNC} with the results of \cite{Schoolmann20}. There it is shown that, for all $N\geq 1$ and all $k\in (0,1]$, 
$$\|S_N\|_{\dinfext\to\dinf}\leq  \frac Ck \left(\frac{\lambda_{N+1}}{\lambda_{N+1}-\lambda_N}\right)^{1/k}.$$
The right hand side is optimal for $k=\frac1 {\log\left(\frac{\lambda_{N+1}}{\lambda_{N+1}-\lambda_N}\right)}$ which implies that 
$$\|S_N\|_{\dinfext\to\dinf}\leq C\log\left(\frac{\lambda_{N+1}}{\lambda_{N+1}-\lambda_N}\right).$$
Hence, we get the case $M=N+1$ of Theorem \ref{thm:BohrNC}.

\subsection{The case of $\mathcal H_\infty(\lambda)$}

So far, we have defined three spaces which are candidates for being the $H_\infty$-space of $\lambda$-Dirichlet series: $\dinf$, $\dinfext$, and $\mathcal H_\infty(\lambda)$. We know that we always have the canonical inclusion $\dinf\subset\dinfext\subset \mathcal H_\infty(\lambda)$ (see \cite[Theorem 2.17]{DSRiesz}) and that, when $\lambda$ satisfies Bohr's theorem, the three spaces are equal. Observe also that $\mathcal H_\infty(\lambda)$ is the only space that is always complete. 

Thus, Theorem \ref{thm:BohrNC} does not always provide an answer for estimating the norm of $S_N$ as an operator on $\mathcal H_\infty(\lambda)$. Fortunately, the proof extends easily using the second (and easiest part) of Theorem \ref{thm:saksman}.

\begin{theorem}\label{thm:BohrNCHinf}
Let  $\lambda$ be a frequency. There exists $C>0$ such that, for all $M>N\geq 1$, 
$$\|S_N\|_{\mathcal H_\infty(\lambda)\to\dinf}\leq C\left(\log\left(\frac{\lambda_M+\lambda_N}{\lambda_M-\lambda_N}\right)+(M-N-1)\right).$$
\end{theorem}
\begin{proof}
We do the proof in $H_\infty^\lambda(G)$. Let $f=\sum_n a_n h_{\lambda_n} \in H_\infty^\lambda(G)$.
We pick the same function $\psi$ and observe, that for almost all $\omega\in G$, 
\begin{align*}
\left|\sum_{n=1}^{+\infty} a_n \widehat\psi(\lambda_n)h_{\lambda_n}(\omega)\right|&\leq \int_{\RR}|f_\omega(it) \psi(t)|dt\\
&\leq \left(C\log\left(\frac{\lambda_M+\lambda_N}{\lambda_M-\lambda_N}\right)+4\right)\|f\|_\infty
\end{align*}
and we conclude as above.
\end{proof}


\section{Maximal inequalities in $H_p^\lambda(G)$}

\subsection{Helson's theorem under (NC)}
In this section, we prove the following theorem, which improves the main result of \cite{DSHelson} and answers an open question of \cite{DSsurvey}.

\begin{theorem}\label{thm:Helson}
Let $\lambda$ satisfy (NC), let $(G,\beta)$ be a $\lambda$-Dirichlet group. For every $u>0$, there exists a constant $C:=C(u,\lambda)$ such that, for all $1\leq p\leq +\infty$ and for all $f\in H_p^{\lambda}(G)$, 
\begin{equation*}
\left\|\sup_{\sigma\geq u}\sup_N \left|\sum_{n=1}^N \hat{f}(h_{\lambda_n})e^{-\sigma \lambda_n}h_{\lambda_n}\right|\right\|_p\leq C\|f\|_p.
\end{equation*}
In particular, for every $u>0$, $\sum_{1}^{+\infty}\hat f(h_{\lambda_n}) e^{-u \lambda_n}h_{\lambda_n}$ converges almost everywhere on $G$. 
\end{theorem}

Let us explain the strategy for the proof. When $p>1$, the almost everywhere convergence is known to hold without any assumption on $\lambda$. This is a consequence of the Carleson-Hunt type result proved in \cite{DSHelson}: for all frequencies $\lambda$, for all $(G,\beta)$ a $\lambda$-Dirichlet group, for all $p\in (1,+\infty)$, there exists $C(p)>0$ such that for all $f\in H_p^\lambda(G)$, 
\begin{equation}
\label{eq:Helson1}
\left( \int_G \sup_n |S_n f(\omega)|^p d\omega\right)^{1/p}\leq C(p)\|f\|_p
\end{equation}
where $S_n(f)=\sum_{k=1}^n \widehat{f}(h_{\lambda_k}) h_{\lambda_k}$ is the partial sum operator (the constant $C(p)$ does not even depend on $\lambda$).  We shall prove a variant of \eqref{eq:Helson1} under (NC), namely
\begin{equation}
\label{eq:Helson2}
\left( \int_G \sup_n e^{-\delta \lambda_n} |S_n f(\omega)|^p d\omega\right)^{1/p}\leq C(\lambda,\delta)\|f\|_p
\end{equation}
valid for all $p\geq 1$, all $\delta>0$ and all $f\in H_p^\lambda(G)$, with a constant $C(\lambda,\delta)$ independent of $p$.

The proof of \eqref{eq:Helson2} will be done for $p=1$ and for $p=+\infty$ and will be finished by interpolation. Unfortunately, it is in general false that $[H_{p_0}^\lambda(G),H_{p_1}^\lambda(G)]_{\theta}=H_{p_\theta}^\lambda(G)$ (see \cite{BayMas}) and we will use an auxiliary operator defined on the whole $L_1(G)$.

We begin by establishing several lemmas. First, we shall prove that we may require additional properties on a sequence satisfying (NC).
\begin{lemma}\label{lem:superNC}
Let $\lambda$ be a frequency satisfying (NC). Then there exists a frequency $\lambda'$ such that $\lambda\subset\lambda'$ and, for all $\delta>0$, there exists $C>0$ such that, for all $n\in\NN$, there exists $m>n$ with 
\begin{eqnarray}
\label{eq:superNC1}
\log(\lambda'_m+\lambda'_n)&\leq&Ce^{\delta\lambda_n'}\\
\label{eq:superNC2}
-\log(\lambda'_m-\lambda'_n)&\leq&Ce^{\delta\lambda'_n}\\
\label{eq:superNC3}
m-n&\leq&Ce^{\delta\lambda'_n}.
\end{eqnarray}
\end{lemma}

\begin{proof}
We construct inductively $\lambda'$ as follows. We set $\lambda_1'=\lambda_1$. Assume that the sequence $\lambda'$ has been built until step $n$, namely that we have constructed $\lambda'_1,\dots,\lambda'_{k_n}$ with $\lambda'_{k_n}=\lambda_n$. If $\lambda_{n+1}\leq\lambda_n+1$, then we set $\lambda'_{k_n+1}=\lambda_{n+1}$ and $k_{n+1}=k_n+1$. Otherwise, we include as many terms $\lambda'_{k_n+1},\dots,\lambda'_{k_{n+1}}$ as necessary so that, for all $j=k_n+1,\dots,k_{n+1}-1$, $1/2\leq \lambda'_{j+1}-\lambda'_j\leq 1$ and $\lambda'_{k_{n+1}}=\lambda_{n+1}$. Namely, we add terms in the sequence $\lambda$ when there is a gap greater than $1$ between two successive terms, and the difference between two consecutive terms of $\lambda'$ is now less than $1$. 

Let us show that the sequence $\lambda'$ satisfies the above conclusion. Let $\lambda'_n$ be any term of the sequence $\lambda'$. If $\lambda'_n$ does not belong to $\lambda$, then we have just to consider $m=n+1$. Otherwise, if $\lambda'_n=\lambda_k$ for some $k\leq n$, there exists $l>k$ such that
\begin{eqnarray}
\log\left(\frac{\lambda_l+\lambda_k}{\lambda_l-\lambda_k}\right)&\leq& Ce^{\delta\lambda_k} \label{eq:superNC4}\\
\nonumber l-k&\leq&Ce^{\delta\lambda_k}.
\end{eqnarray}
Set $m=n+(l-k)$ and observe that we have
$$\log(\lambda'_m+\lambda'_n)\leq \log\left(C e^{\delta\lambda'_n}+2\lambda'_n\right)\leq C'e^{\delta\lambda'_n}.$$
If there is no gap between $\lambda_k$ and $\lambda_l$, then $\lambda'_m=\lambda_l$ and \eqref{eq:superNC4} implies
$$-\log(\lambda'_m-\lambda'_n)=-\log(\lambda_l-\lambda_k)\leq Ce^{\delta\lambda_k}=Ce^{\delta\lambda'_n}.$$
If there is a gap between $\lambda_k$ and $\lambda_l$, then $\lambda'_m-\lambda'_n\geq 1/2$, and \eqref{eq:superNC2} holds trivially.
\end{proof}

In the sequel, when we will pick a frequency $\lambda$ satisfying (NC), we will in fact assume that it satisfies the stronger properties given by Lemma \ref{lem:superNC}.

\smallskip

For $a>0$ and $h>0$, we shall denote by $\psi_{a,h}$ the function defined by 
$$\psi_{a,h}(t)=\frac{\sin((a+h)t)}{t}\times\frac{\sin(ht)}{ht}.$$
The estimation of the $L^1$-norm of $\psi_{a,h}$ was a crucial point in order to apply Saksman's convolution formula during the proof of Theorem \ref{thm:BohrNC}. In order to obtain our maximal estimates, we will need a similar inequality allowing now $a$ and $h$ to vary.

\begin{lemma}\label{lem:L1norm}
Let $a:\RR\to(0,+\infty)$ and $h:\RR\to(0,+\infty)$ be two measurable functions. Assume that there exists $\kappa>0$ such that $a(t)+h(t)\leq\kappa$ and $h(t)\geq \kappa^{-1}$ for all $t\in\RR$. Then 
$$\int_{\RR}\left|\psi_{a(t),h(t)}(t)\right|dt\leq 4+4\log\kappa.$$
\end{lemma}
\begin{proof}
It suffices to observe that 
\begin{itemize}
\item when $0<|t|\leq\kappa^{-1}$, then 
$$\left|\psi_{a(t),h(t)}(t)\right|\leq |a(t)+h(t)|\times 1\leq\kappa.$$
\item when $\kappa^{-1}\leq|t|\leq \kappa$, then 
$$\left|\psi_{a(t),h(t)}(t)\right|\leq \frac1{|t|}\times 1=\frac 1{|t|}.$$
\item when $|t|\geq\kappa$, then 
$$\left|\psi_{a(t),h(t)}(t)\right|\leq \frac{1}{h(t) t^2}\leq\frac{\kappa}{t^2}.$$
\end{itemize}
\end{proof}

We now fix a frequency $\lambda$ satisfying (NC) and $\delta>0$. Let $C>0$ and $m:\NN\to\NN$ be such that $m(n)>n$  for all $n\in\NN$ and \eqref{eq:superNC1}, \eqref{eq:superNC2}, \eqref{eq:superNC3} are satisfied for $m=m(n)$. For $n\in\mathbb N$, we shall denote by $h_n=(\lambda_{m(n)}-\lambda_n)/2$ and by $\phi_n$ the function $\phi_n=\psi_{\lambda_n,h_n}$.  Let us recall that $R_{\phi_n}$ is defined on $H_1^{\lambda}(G)$ by

\begin{equation}\label{eq:Helson4}
R_{\phi_n}(f)=\sum_k \widehat f(h_{\lambda_k})\widehat{\phi_n}(\lambda_k)h_{\lambda_k}.
\end{equation}
By the vertical convolution formula, we also know that $R_{\phi_n}$ is given by, for a.e. $\omega\in G$,
\begin{equation}\label{eq:Helson5}
R_{\phi_n}(f)(\omega)=\int_\RR f(\omega\beta(t))\phi_n(t)dt.
\end{equation}
Now the right hand side of the previous equality is well-defined for all functions in $L_1(G)$. Thus we will think at $R_{\phi_n}$ as the operator on $L_1(G)$ defined by \eqref{eq:Helson5}, keeping in mind that it also verifies \eqref{eq:Helson4} for $f\in H_1^\lambda(G)$. In this context, we shall prove the following maximal inequality on $R_{\phi_n}$:

\begin{lemma}\label{lem:maximalRphi}
For all $\delta>0$, there exists $C>0$ such that, for all $p\in [1,+\infty]$, for all $N\in\mathbb N$, for all $f\in L_p(G)$, 
$$\left(\int_{G}\sup_{n\leq N}|R_{\phi_n}f(\omega)|^p d\omega\right)^{1/p} \leq Ce^{\delta\lambda_N}\|f\|_p.$$
\end{lemma}
\begin{proof}
We start with the case $p=1$. It is enough to prove it for $f\in\mathcal C(G)$. 
Define $n:G\to\{1,\dots,N\}$, $\omega\mapsto n(\omega)$ by
$$n(\omega)=\inf\left\{l\in\{1,\dots,N\}:\ |R_{\phi_l}f(\omega)|=\sup_{n\leq N}|R_{\phi_n}f(\omega)|\right\}.$$
The function $n$ is measurable and

\begin{align*}
\int_G \sup_{n\leq N}|R_{\phi_n}(f)(\omega)|d\omega&=
\int_G|R_{\phi_{n(\omega)}}(f)(\omega)|d\omega\\
&\leq \int_{\RR}\int_G |f(\omega \beta(t))|\cdot  | \psi_{\lambda_{n(\omega)},h_{n(\omega)}}(t)|d\omega dt.
\end{align*}
In the inner integral we do the change of variables $\omega'=\omega\beta(t)$ so that 
\begin{align*}
\int_G \sup_{n\leq N}|R_{\phi_n}(f)(\omega)|d\omega&\leq 
\int_{\RR}\int_G |f(\omega')|\cdot | \psi_{\lambda_{n(\omega' \cdot \beta(t)^{-1})},h_{n(\omega'\cdot \beta(t)^{-1})}}(t)|d\omega dt\\
&\leq \int_G |f(\omega')|\int_{\RR}  | \psi_{\lambda_{n(\omega' \cdot \beta(t)^{-1})},h_{n(\omega'\cdot \beta(t)^{-1})}}(t)|dtd\omega .
\end{align*}
We now use Lemma \ref{lem:L1norm} together with \eqref{eq:superNC1} and \eqref{eq:superNC2}. This yields
\begin{align*}
\int_G \sup_{n\leq N}|R_{\phi_n}(f)(\omega)|d\omega&\leq C\int_G |f(\omega')|e^{\delta\lambda_N}d\omega'=Ce^{\delta\lambda_N}\|f\|_1.
\end{align*}

We then do the case $p=+\infty$. Let $f\in L_\infty(G)$. Then 
\begin{align*}
\sup_{\omega\in G}\sup_{n\leq N}|R_{\phi_n}f(\omega)|&=\sup_{n\leq N}\sup_{\omega\in G}
\int_{\RR} |f(\omega\beta(t))|\cdot |\psi_{\lambda_n,h_n}(t)|dt\\
&\leq \sup_{n\leq N} \|\psi_{\lambda_n,h_n}\|_1\|f\|_\infty\\
&\leq Ce^{\delta\lambda_N}\|f\|_\infty.
\end{align*}
We then conclude by interpolation.
\end{proof}

We deduce from the above work a weighted Carleson-Hunt maximal inequality for $H_1^{\lambda}(G)$-functions, which seems interesting for itself when $p=1$ (for $p\in (1,+\infty)$, an unweighted Carleson-Hunt inequality is true, the point here is that the constant does not depend on $p$). This statement was inspired by \cite{BAYHEUR2} where a similar result in the (much easier) case of $H_1(\TT)$ was essential to do a multifractal analysis of the divergence of Fourier series of functions of $H_1(\TT)$.

\begin{theorem}\label{thm:CarlesonHunt}
Let $\lambda$ satisfying (NC). For all $\delta>0$ there exists $C>0$ such that, for all $N\in\NN$, for all $p\geq1$, for all $f\in H_p^\lambda(G)$,
$$\left(\int_G \sup_{n\leq N} |S_n f(\omega)|^p d\omega\right)^{1/p}\leq Ce^{\delta\lambda_N}.$$
\end{theorem}
\begin{proof}
We argue as in the proof of Theorem \ref{thm:BohrNC}, namely we write for a fixed $n\in\NN$, 
$$|S_n f(\omega) | \leq |R_{\phi_n}f(\omega)|+m(n)-n\leq |R_{\phi(n)}f(\omega)|+Ce^{\delta\lambda_n}.$$
Therefore, 
$$\sup_{n\leq N}|S_n f(\omega)|\leq \sup_{n\leq N}|R_{\phi_n}(f)(\omega)|+Ce^{\delta\lambda_N}$$
and we conclude by taking the $L_p(G)$-norm.
\end{proof}

We are now ready for the proof of Theorem \ref{thm:Helson}.

\begin{proof}[Proof of Theorem \ref{thm:Helson}]
We first proceed with the case $p\in [1,+\infty)$. We may assume that $f\in Pol_\lambda(G)$. Let $\delta=u/3$. 
For $\sigma\geq u$, using Lemma 3.4 of \cite{DSHelson}, we have
$$\left|\sum_{n=1}^N \widehat f(h_{\lambda_n}) e^{- \sigma\lambda_n} h_{\lambda_n}\right|^p\leq 
C(u)^p \sup_{n\leq N} \left| e^{-2\delta\lambda_n } \sum_{k=1}^n \widehat f(h_{\lambda_k}) h_{\lambda_k}\right|^p.$$
Hence, 
$$\sup_{\sigma>u}\sup_N \left|\sum_{n=1}^N \widehat f(h_{\lambda_n}) e^{-\sigma\lambda_n } h_{\lambda_n}\right|^p \leq C(u)^p \sup_N \left|e^{-2\delta\lambda_N } \sum_{n=1}^N \widehat f(h_{\lambda_n})h_{\lambda_n}\right|^p.$$
For $\omega\in G$, we define
$$n(\omega)=\inf\left\{l\geq 0:\ \left|e^{-2\delta\lambda_l}\sum_{n=1}^l \widehat f(h_{\lambda_n})h_{\lambda_n}\right|=\sup_N \left|e^{-2\delta\lambda_N}\sum_{n=1}^N \widehat f(h_{\lambda_n})h_{\lambda_n}\right|\right\}$$
which is measurable. For $k\geq 0$, we set
$$A_k=\left\{n:\ \lambda_n\in[k,k+1)\right\},\ G_k=\left\{\omega\in G:\ n(\omega)\in A_k\right\},$$
$$I(\sigma)=\int_G \sup_{\sigma>u}\sup_N \left|\sum_{n=1}^N \widehat f(h_{\lambda_n}) e^{-\sigma\lambda_n } h_{\lambda_n}\right|^p d\omega.$$
We can write
\begin{align*}
I(\sigma)&\leq C(u)^p \sum_{k\geq 0}\int_{G_k}\sup_N \left|e^{-2\delta\lambda_N }\sum_{n=1}^N \widehat f(h_{\lambda_n})h_{\lambda_n}\right|^p d\omega\\
&\leq C(u)^p \sum_{k\geq 0}\int_{G_k}\sup_{N\in A_k} \left|e^{-2\delta\lambda_N}\sum_{n=1}^N \widehat f(h_{\lambda_n})h_{\lambda_n}\right|^p d\omega\\
&\leq C(u)^p \sum_{k\geq 0}\int_{G_k}e^{-2 \delta p k}\sup_{N\in A_k} \left|\sum_{n=1}^N \widehat f(h_{\lambda_n})h_{\lambda_n}\right|^p d\omega\\
&\leq C(u,\lambda)^p \sum_{k\geq 0}e^{-2\delta  p k} e^{\delta p(k+1) }\|f\|_p^p \\
&\leq C(u,\lambda)^p \|f\|_p^p.
\end{align*}
As for the proof of Lemma \ref{lem:maximalRphi}, the proof is easier for $p=\infty$ and is left to the reader.

\end{proof}

If we analyze the previous proof carefully, we observe that we have obtained the following (slightly stronger) variant of Theorem \ref{thm:CarlesonHunt}.
\begin{corollary}
Let $\lambda$ satisfying (NC). For all $\delta>0$, there exists $C>0$ such that, for all $p\geq1$, for all $f\in H_p^\lambda(G)$, 
$$\left(\int_G \sup_N \left|\frac{S_N f(\omega)}{e^{\delta\lambda_N}}\right|^p d\omega\right)^{1/p}\leq C(\delta)\|f\|_p.$$
\end{corollary}

When $\lambda$ satisfies (BC), it is possible to improve this inequality.
\begin{proposition}
Let $\lambda$ satisfy (BC). For all $\alpha>1$, there exists $C>1$ such that, for all $p\geq 1$, for all $f\in H^\lambda_p(G)$, 
$$\left(\int_G\sup_N \left|\frac{S_N f(\omega)}{\lambda_N^\alpha}\right|^p d\omega\right)^{1/p} \leq C(\alpha) \|f\|_p.$$
\end{proposition}
\begin{proof}
We just sketch the proof. If $\lambda$ satisfies (BC), then we know that there exists $C>0$ such that, for all $n\in\NN$, $\log(\lambda_{n+1}-\lambda_n)\geq -C\lambda_n$. Adding terms if necessary, we can also assume that $\log(\lambda_{n+1}+\lambda_n)\leq C\lambda_n$. Arguing exactly as in the proof of Theorem \ref{thm:CarlesonHunt}, we can prove the existence of $C>0$ such that, for all $f\in H_p^\lambda(G)$, for all $n\in\NN$, 
$$\int_G \sup_{n\leq N}|S_nf(\omega)|^p d\omega\leq C\lambda_N^p.$$
Let now $\alpha>1$, fix $f\in \textrm{Pol}_\lambda(G)$ and define, for $\omega\in G$, 
\begin{align*}
n(\omega)&=\inf\left\{ l\geq 0:\ \left|\lambda_l^{-\alpha}\sum_{n=1}^l \widehat f(h_{\lambda_n}) h_{\lambda_n}\right|=\sup_N \left|\lambda_N^{-\alpha}\sum_{n=1}^N \widehat f(h_{\lambda_n}) h_{\lambda_n}\right|\right\}\\
A_k&=\{n:\ \lambda_n\in [2^k,2^{k+1})\}\\
G_k&=\{w:\ n(\omega)\in A_k\}.
\end{align*}
Then
\begin{align*}
\int_G \sup_N \left|\frac{S_N f(\omega)}{\lambda_N^\alpha}\right|^p d\omega&=\sum_k \int_{G_k} \sup_N \left|\frac{S_N f(\omega)}{\lambda_N^\alpha}\right|^p d\omega\\
&=\sum_k \int_{G_k} \sup_{N\in A_k} \left|\frac{S_N f(\omega)}{\lambda_N^\alpha}\right|^p d\omega\\
&\leq \sum_k 2^{-pk\alpha}\int_G \sup_{\lambda_N\leq 2^{k+1}}|S_N f(\omega)|^p d\omega\\
&\leq C\sum_k 2^{-pk\alpha}2^{p(k+1)}\|f\|_p^p.
\end{align*}
\end{proof}

\begin{question}
We know that $\lambda$ satisfies Bohr's theorem if and only if for all $\delta>0$, there exists $C>0$ such that, for all $f\in H_\infty^\lambda(G)$, for all $N\geq 1$,
\begin{equation}\label{eq:q1}
\left\|\sup_{n\leq N} \left|\sum_{k=1}^n \widehat f(h_{\lambda_k}) h_{\lambda_k}(\omega)\right|\ \right\|_{L_\infty(G)}\leq Ce^{\delta\lambda_N}\|f\|_\infty.
\end{equation}
We have shown that if $\lambda$ satisfies (NC), then it satisfies the previous inequality.
To prove that Helson's theorem is satisfied (and even to prove that the relevant maximal inequality holds true), it is sufficient to prove that, for all $\delta>0$, there exists $C>0$ such that, for all $f\in H_1^\lambda(G)$, for all $N\geq 1$,
\begin{equation}
\label{eq:q2}
\left\|\sup_{n\leq N} \left|\sum_{k=1}^n \widehat f(h_{\lambda_k}) h_{\lambda_k}(\omega)\right|\ \right\|_{L_1(G)} \leq Ce^{\delta\lambda_N}\|f\|_1.
\end{equation}
Again we have shown that if $\lambda$ satisfies (NC), then \eqref{eq:q2} is true. It seems natural to ask whether
\eqref{eq:q2} always follows from \eqref{eq:q1} or, equivalently, if any frequency $\lambda$ satifying Bohr's theorem also satisfies Helson's theorem. Inequalities in $\mathcal H_1(\lambda)$ have already been deduced for their vector-valued counterpart in $\mathcal H_\infty(\lambda)$ in \cite{DS19}. At first glance, it seems that this argument cannot be applied here.
\end{question}

\subsection{Failure of Helson's theorem for $p=1$}
Since for $p>1$, for any frequency $\lambda$, for any $(G,\beta)$ a $\lambda$-Dirichlet group, for any $g\in H_p^\lambda(G)$, the series $\sum_{n=1}^{+\infty}\hat f(h_{\lambda_n})h_{\lambda_n}$ converges almost everywhere on $G$ (this follows from the Carleson-Hunt theorem of \cite{DSHelson}), it is natural to ask whether Theorem \ref{thm:Helson} remains true without any assumption on $\lambda$. We show that this is not the case.
\begin{theorem}\label{thm:failureHelson}
There exists a frequency $\lambda$, a $\lambda$-Dirichlet group $(G,\beta)$ and $f\in H_1^\lambda(G)$ such that, for all $u>0$, the series $\sum_{n=1}^{+\infty}\hat f(h_{\lambda_n})e^{-u\lambda_n}h_{\lambda_n}$ diverges almost everywhere on $G$. 
\end{theorem}
As we might guess, the proof will use the results of Kolmogorov on a.e. divergent Fourier series in $L_1(\TT)$ (see for instance \cite{Zygmund}).

\begin{lemma}\label{lem:Kolmogorov}
Let $A,\delta>0$. There exists $P\in H_1(\TT)$ a polynomial and $E\subset\TT$ measurable such that 
\begin{itemize}
\item $\|P\|_1\leq\delta$.
\item $m_\TT(E)\geq 1-\delta$ (here, $m_\TT$ denotes the Lebesgue measure on $\TT$).
\item for all $z\in E$, there exists $n(z)\in\NN$ such that $|S_{n(z)} P(z)|\geq A$.
\end{itemize}
\end{lemma}

\begin{proof}
By induction on $j\geq 1$, we construct a sequence of holomorphic polynomials $(P_j)$ with $\deg(P_j)=d_j$, two sequences of positive real numbers $(\mu_j)$ and $(\veps_j)$ and a sequence $(E_j)$ of measurable subsets of $\TT$ such that the following properties are true for each $j$:
\begin{enumerate}[(a)]
\item $m_\TT(E_j)\geq 1-2^{-j}$
\item $\|P_j\|_1\leq 2^{-j}$
\item $\mu_j>\mu_{j-1}+d_{j-1}\veps_{j-1}$
\item the real numbers $2\pi, \mu_1,\dots,\mu_j,\veps_1,\dots,\veps_j$ are $\QQ$-independent
\item for each $z\in E_j$, there exists an integer $n_j(z)$ such that, for all $u\in [0,j]$, 
$$\left|\sum_{k=0}^{n_j(z)} \widehat {P_j}(k) e^{-uk\veps_j}z^k\right|\geq je^{j\mu_j}.$$
\end{enumerate}
Let us proceed with the construction. We choose for $\mu_j$ any real number such that $\mu_j>\mu_{j-1}+d_{j-1}\veps_{j-1}$ and the real numbers $2\pi,\mu_1,\dots,\mu_j,\veps_1,\dots,\veps_{j-1}$ are independent over $\QQ$ (when $j=1$, we simply choose $\mu_1>2\pi$ with $(2\pi,\mu_1)$ independent over $\QQ$). We then apply Lemma \ref{lem:Kolmogorov} with $A=(j+1)e^{j\mu_j}$ and $\delta=2^{-j}$ to get a polynomial $P_j$ with degree $d_j$ and a subset $E_j\subset\TT$ satisfying (a) and (b). Since the functions $(u,z)\mapsto \sum_{k=0}^{n} \widehat P_j(k) e^{-uk\veps}z^k$, for $0\leq n\leq d_j$, converge uniformly on $[0,j]\times\TT$ to $(u,z)\mapsto S_n P(z)$ as $\veps\to 0$, we may choose $\veps_j$ a sufficiently small positive real number such that (d) and (e) are satisfied.

Define now $\lambda=\{\mu_j+k\veps_j:\ j\geq 1,\ 0\leq k\leq d_j\}$, $G=\prod_{j=1}^{+\infty}\TT^2$ endowed with the canonical product structure and define $\beta:(\RR,+)\to G,\ t\mapsto \big( (e^{-it\mu_j},e^{-it\veps_j})\big)_j$. By (d) and Kronecker's theorem, the homomorphism $\beta$ has dense range. Moreover, let $\lambda_n\in\lambda$. Then $\lambda_n=\mu_j+k\veps_j$ for some $j\geq 1$ and some $0\leq k\leq d_j$. Write an element $\omega\in G$ as $\prod_{l=1}^{+\infty}(w_l,z_l)$ and define $h_{\lambda_n}(\omega)=w_j z_j^k$. Then 
$$h_{\lambda_n}\circ\beta(t)=e^{-it(\mu_j+k\veps_j)}=e^{-i\lambda_n t}$$
so that $(G,\beta)$ is a $\lambda$-Dirichlet group. Now, define
$$f_j(\omega)=w_j P_j(z_j)=\sum_{k=0}^{d_j} \widehat{P_j}(k)h_{{\mu_j+k\veps_j}}(\omega).$$
We get $\|f_j\|_{H_1^{\lambda}(G)}=\|P_j\|_{H_1(\TT)}$ so that the series $f=\sum_{j\geq 1}f_j$ converges in $H_1^\lambda(G)$. Let us also define
$$F_j=\{\omega\in G:\ z_j\in E_j\}.$$
Then $m_G(F_j)=m_{\TT}(E_j)\geq 1-2^{-j}$ (here, $m_G$ denotes the Haar measure on $G$). Thus, if we set $F=\bigcap_{j_0\geq 1}\bigcup_{j\geq j_0}F_j$, then $m_G(F)=1$. Pick now $\omega\in F$. We may find $j$ as large as we want such that $\omega\in F_j$. The construction of $P_j$ ensures that there exists $0\leq n_j(z_j)\leq d_j$ such that, for all $u\in [0,j]$, 
$$\left|\sum_{k=0}^{n_j(z_j)} \widehat{P_j}(k) e^{-u(\mu_j+k\veps_j)}z_j^k \right|\geq j.$$
Setting $N$, resp. $M$, such that $\lambda_N=\mu_{j-1}+d_{j-1}\veps_{j-1}$, resp. $\lambda_M=\mu_j+n_j(z_j)\veps_j$, the previous inequality translates into
$$\left|\sum_{n=N+1}^{M} \hat f(h_{\lambda_n}) e^{-u\lambda_n} h_{\lambda_n}(\omega)\right|\geq j.$$
This easily yields the a.e. divergence of $\sum_{n=1}^{+\infty}\widehat f(h_{\lambda_n}) e^{-u\lambda_n} h_{\lambda_n}$.
\end{proof}

\subsection{Maximal inequalities for mollifiers}

Since $(S_n f)$ does not necessarily converge pointwise or even in norm for all functions in $H_1^\lambda(G)$, Defant and Schoolmann looked in \cite{DSRiesz} for a substitute by changing the summation method. They succeeded by choosing Riesz means.  Precisely they showed (see \cite[Theorem 2.1]{DSRiesz}), through a maximal inequality, that for all frequencies $\lambda$, for all $f\in H_\lambda^1(G)$, for all $\alpha>0$, the sequence $(R_N^{\lambda,\alpha}(f)(\omega))$ converges to $f(\omega)$ for almost all $\omega\in G$. We extend this to a large class of mollifiers.

\begin{theorem}\label{thm:maxRiesz}
Let $\lambda$ be a frequency,  $(G,\beta)$ a $\lambda$-Dirichlet group. 
Let $\psi\in L_1(\RR)$ be a continuous function (except at a finite number of points) such that $\widehat \psi$ has compact support and there exists a nonincreasing function $g\in L_1(0,+\infty)$ such that $|\psi(x)|\leq g(|x|)$ for all $x\in\RR$. For $N\geq 1$, define $\psi_N(\cdot)=\psi(\cdot/N)$. Then 
$$R_{\max,\psi}(f):=\sup_N |R_{\psi_N}(f)|$$
defines a bounded sublinear operator from $H_1^\lambda(G)$ into $L_{1,\infty}(G)$. Moreover, if $\int \psi=1$, then for all $f\in H_1^\lambda(G)$, $R_{\psi_N}(f)(\omega)$ converges for almost every $\omega\in G$ to $f(\omega)$.
\end{theorem}

\begin{proof}
Again, the key point is the vertical convolution formula. Indeed, we know that for a.e. $\omega\in G$,
$$R_{\psi_N}(f)(\omega)=f_\omega\star \psi_N(0).$$
For those $\omega$, using \cite[Theorem 2.1.10 and Remark 2.1.11]{Gramodern},
\begin{align*}
\sup_N |R_{\psi_N}(f)(\omega)|&\leq \sup_N |f_\omega|\star \psi_N(0)\\
&\leq 2\| g\|_1 \overline{M}f(\omega)
\end{align*}
where $\overline M(f)(\omega)=\sup_{T>0}\frac1{2T}\int_{-T}^T |f_\omega(t)|dt$ is the appropriate Hardy-Littlewood maximal operator. Since $\overline M$ maps $H_1^\lambda(G)$ into $L_{1,\infty}(G)$ by \cite[Theorem 2.10]{DSRiesz}, we can conclude about the first assertion of the theorem. The result on a.e. convergence is then a standard corollary of it, using that it is clearly true for polynomials since $\widehat \psi(0)=1$.
\end{proof}
\begin{remark}
We can replace the assumption that $\psi$ is compactly supported by the assumption that, for all $N\geq 1$, $\sum_n |\widehat{\psi}(\lambda_n/N)|<+\infty$.
\end{remark}

This last theorem covers many examples. For instance, for all $0\leq a<b$, we may choose the function $\psi\in L_1(\RR)$ such that $\hat\psi=1$ on $[-a,a]$, $\hat\psi=0$ on $(-\infty,-b)\cup(b,+\infty)$ and $\hat \psi$ is affine on $(-b,-a)$ and on $(a,b)$.   As already observed during the proof of Theorem \ref{thm:BohrNC}, the function $\psi$ is given by 
\[ \psi(x)=C(a,b)\frac{\sin\left(\frac{a+b}2x\right)\sin\left(\frac{b-a}2x\right)}{x^2} \]
which clearly satisfies the assumptions of Theorem \ref{thm:maxRiesz}. This is also the case for
$\psi(x)=e^{-|x|}$ or $\psi(x)=e^{-x^2}$, provided the frequency $\lambda$ satisfies $\sum |\widehat{\psi}(\lambda_n/N)|<+\infty$ for all $N\geq 1$. To show that our result covers Theorem 2.1 of \cite{DSRiesz}, we also have to show that for $\alpha>0$ the function $\psi\in L_1(\RR)$ that satisfies
$$\widehat{\psi}(t)=(1-|t|)^\alpha \mathbf1_{[-1,1]}(t)$$
verifies the assumptions of Theorem \ref{thm:maxRiesz}. Let $x>0$. We already have observed that
$$\psi(x)=\frac1{ix}\mathcal F\left(\pm\alpha (1-|t|)^{\alpha-1}\mathbf 1_{[-1,1]}\right)(x).$$
Fix $\beta>0$ such that $|\beta(\alpha-1)|<1$ and let $x\geq 1$. Then
\begin{align*}
|\psi(x)|&\leq\frac{2\alpha}x \left |\int_0^1 (1-t)^{\alpha-1}e^{itx}dt\right|\\
&\leq \frac{2\alpha}x \left|\int_0^1 u^{\alpha-1}e^{-iux}du\right|.
\end{align*}
We split the integral into two parts. First, 
\[ \left| \int_0^{x^{-\beta}} u^{\alpha-1} e^{-iux}du\right|\leq \frac 1\alpha x^{-\alpha\beta}.\]
Second, integrating by parts,
\[ \int_{x^{-\beta}}^1 u^{\alpha-1}e^{-iux}du=\frac{-1}{ix}\left[u^{\alpha-1}e^{-iux}\right]_{x^{-\beta}}^1
+\frac{\alpha-1}{ix}\int_{x^{-\beta}}^1 u^{\alpha-2}e^{-iux}du \]
so that
$$\left |\int_{x^{-\beta}}^1 u^{\alpha-1} e^{-iux}du\right|\leq C\left(\frac 1x+\frac1{x^{1+(\alpha-1)\beta}}\right).$$
Our choice of $\beta$ guarantees that there is $\delta>0$ and $C>0$ such that, for $x\geq 1$, 
$$|\psi(x)|\leq \frac{C}{x^{1+\delta}}.$$
This shows that the assumptions of Theorem \ref{thm:maxRiesz} are satisfied with
$$g(x)=\left\{
\begin{array}{ll}
\displaystyle \frac C{x^{1+\delta}},&x\geq 1\\
\max(\|\psi\|_\infty,C),&x\in[0,1).
\end{array}\right.$$

\begin{remark}
Lemma \ref{lem:maximalRphi} and Theorem \ref{thm:maxRiesz} are of course very close. The latter one is true for all frequencies $\lambda$, but we start from a fixed function $\psi$ and it does not cover fully the case $p=1$. Lemma \ref{lem:maximalRphi} adapts at each step the $L^1$-function to the frequency $\lambda$ and to the function $f$. The price to pay is that we lose some factor $e^{\delta\lambda_N}$ and that we cannot use general results on the Hardy-Littlewood maximal function.
\end{remark}


\section{Horizontal translations}

In this section, we investigate the boundedness from $\mathcal H_p(\lambda)$ into $\mathcal H_q(\lambda)$, for $q>p$, of the horizontal translation map $T_\sigma(\sum_n a_n e^{-\lambda_n s})=\sum_n a_n e^{-\sigma \lambda_n}e^{-\lambda_n s}$. We are interested in this map to determine the exact value of 
$$\sigma_{\mathcal H_p(\lambda)}=\inf\{\sigma\in\RR:\ \sigma_c(D)\leq \sigma\textrm{ for all }D\in\mathcal H_p(\lambda)\}$$
since it is easy to prove, using the Cauchy-Schwarz inequality, that $\sigma_{\mathcal H_2(\lambda)}=L(\lambda)/2$. 
Recall that, when $\lambda=(\log n)$, $\sigma_{\mathcal H_p(\lambda)}=1/2$ for all $p\in[1,+\infty)$. In the general case,
it is always possible to majorize $\sigma_{\mathcal H_1(\lambda)}$ if we know $\sigma_{\mathcal H_2(\lambda)}$.

\begin{proposition}\label{prop:sigmah1}
Let $\lambda$ be a frequency. Then $\sigma_{\mathcal H_1(\lambda)}\leq 2\sigma_{\mathcal H_2(\lambda)}$. 
\end{proposition}
\begin{proof}
Let $\veps>0$. It is sufficient to prove that, for all $f=\sum_j a_j e^{-\lambda_j s}$ belonging to $\mathcal H_1(\lambda)$,
for all $\sigma>2\sigma_{\mathcal H_2(\lambda)}+\veps=L(\lambda)+\veps$, 
$$\sum_{j=1}^{+\infty}|a_j|e^{-\lambda_j \sigma}<+\infty.$$
Let $J\geq 1$ be such that, for all $j\geq J$, $\log(j)/\lambda_j\leq L(\lambda)+\veps$. Then
$$\sum_j |a_j|e^{-\lambda_j \sigma}\leq (J-1) \|f\|_1+\sum_{j=J}^{+\infty}\|f\|_1 e^{-\frac{\sigma}{L(\lambda)+\veps}\log(j)}<+\infty$$
by our assumption on $\sigma$. 
\end{proof}

It turns out that, even if we put strong growth and separation conditions on $\lambda$, we cannot go further.

\begin{theorem}\label{thm:H1H2}
There exists a frequency $\lambda$ satisfying (BC) such that $\sigma_{\mathcal H_1(\lambda)}=2\sigma_{\mathcal H_2(\lambda)}$ and $\sigma_{\mathcal H_2(\lambda)}>0$.
\end{theorem}
\begin{proof}
For $n\geq 2$, let $\delta_n\in (2^{-n-1},2^{-n}]$ such that $(2\pi,n,\delta_n)$ are $\mathbb Z$-independent. We set 
$$\lambda_{2^n+k}=n+k\delta_n\textrm{ for }n\geq 1,\ k=0,\dots,2^n-1.$$
It is easy to check that $L(\lambda)=\log(2)$ so that $\sigma_{\mathcal H_2(\lambda)}=(\log 2)/2$. Moreover, it is
also easy to check that $\lambda$ satisfies (BC). Indeed, for $n\geq 2$ and $k=0,\dots,2^n-2$, 
$$\lambda_{2^n+k+1}-\lambda_{2^n+k}=\delta_n \geq\frac 12 2^{-n}\geq Ce^{-(\log 2)\lambda_{2^n+k}}$$
and similarly 
$$\lambda_{2^{n+1}}-\lambda_{2^{n+1}-1}\geq 2^{-n}\geq Ce^{-(\log 2)\lambda_{2^{n+1}-1}}.$$
 Pick now any $\sigma>\sigma_{\mathcal H_1(\lambda)}$. By the principle of uniform boundedness, there exists $C_0>0$ such that, for 
 all $D=\sum_j a_j e^{-\lambda_j s}$ belonging to $\mathcal H_1(\lambda)$, for all $N\geq 2$, 
 \begin{equation}
  \left |\sum_{j=2}^N a_j e^{-\lambda_j \sigma}\right| \leq C_0 \|D\|_1 \label{eq:dif1}
 \end{equation}
Let $n\geq 2$ and choose $D=\sum_{k=0}^{2^n-1}e^{-\lambda_{2^n+k}s}$. Then 
\begin{equation}\label{eq:dif2}
\sum_{k=0}^{2^n-1}e^{-\lambda_{2^n+k}\sigma}\geq 2^n e^{-\sigma(n+1)}\geq C_1 e^{(\log 2-\sigma)n}.
\end{equation}
On the other hand, set $\lambda'=\{n+k\delta_n:k\geq 0\}$ and observe that, using the internal description of the norm of $\mathcal H_1$, 
$$\|D\|_{\mathcal H_1(\lambda)}=\|D\|_{\mathcal H_1(\lambda')}.$$
We shall compute $\|D\|_{\mathcal H_1(\lambda')}$ using Fourier analysis. Indeed, since $(2\pi,n,\delta_n)$ are $\mathbb Z$-independent, the map $\beta:\mathbb R\to\mathbb T^2,\ t\mapsto(e^{-itn},e^{-it\delta_n})$ has dense range, so that 
$(\mathbb T^2,\beta)$ is a $\lambda'$-Dirichlet group. Therefore, 
\begin{equation}
\|D\|_{\mathcal H_1(\lambda')}=\int_{\mathbb T^2}\left|\sum_{k=0}^{2^n-1}z_1z_2^k\right|dz_1dz_2\leq C_2 n
\label{eq:dif3}
\end{equation}
by the classical estimate of the norm of the Dirichlet kernel. Hence, \eqref{eq:dif1}, \eqref{eq:dif2} and \eqref{eq:dif3} imply that, for all $\sigma>\sigma_{\mathcal H_1(\lambda)}$, $\sigma\geq\log(2)$. This yields $\sigma_{\mathcal H_1(\lambda)}\geq 2\sigma_{\mathcal H_2(\lambda)}$.
\end{proof}

In view of the previous results, it seems natural to study how arithmetical properties of the frequency $\lambda$ can influence the values of $\sigma$ for which $T_\sigma:\mathcal H_p(\lambda)\to \mathcal H_q(\lambda)$, $p<q$, is bounded. 
We concentrate on the case $p=2$ and $q=2k$, $k\geq 1$, because we can compute the norms using the coefficients. We define $\lambda*\lambda$ as $\{\lambda_l+\lambda_k:\ l,k\geq 0\}$ and $\lambda^{*k}=\lambda*\cdots*\lambda$ (with $k$ factors).


\begin{definition}
Let $\lambda$ be a frequency and $k\geq 1$. Write $\lambda^{*k}=(\mu_l)$ where the sequence $(\mu_l)$ is increasing. We set
$$A(\lambda,k)=\limsup_{l\to+\infty}\frac{\log\left(\card\left\{(n_1,\dots,n_k):\lambda_{n_1}+\cdots+\lambda_{n_k}=\mu_l\right\}\right)}{2\mu_l}.$$
\end{definition}

\begin{proposition}\label{prop:H2H2K}
 Let $\lambda$ be a frequency and $k\geq 1$. Then for $\sigma>A(\lambda,k)$,  $T_\sigma$ maps boundedly $\mathcal H_2(\lambda)$ into $\mathcal H_{2k}(\lambda)$. 
\end{proposition}
\begin{proof}
We shall prove a slightly stronger statement : if $\sigma>0$ is such that there exists $C>0$ such that for any $\mu>0$, 
 $$e^{-2\mu\sigma}\textrm{card}\{(n_1,\dots,n_k): \lambda_{n_1}+\cdots+\lambda_{n_k}=\mu\}\leq C,$$
then $T_\sigma$ maps boundedly $\mathcal H_2(\lambda)$ into $\mathcal H_{2k}(\lambda)$. 
 Indeed, let $D=\sum_n a_n e^{-\lambda_n s}$ belonging to $\mathcal H_2(\lambda)$. We write 
 $T_\sigma(D)^k=\sum_l b_{\mu_l}e^{-\mu_l s}$ where $$b_{\mu_l}=\sum_{\lambda_{n_1}+\cdots+\lambda_{n_k}=\mu_l}a_{n_1}\cdots  a_{n_k} e^{-\mu_l\sigma}.$$
 We just need to prove that the sequence $(b_{\mu_l})$ is square summable, namely that for all square summable sequences $(c_{\mu_l})$, 
 $\sum_l b_{\mu_l}c_{\mu_l}$ is convergent, namely that 
 $$\sum_{n_1,\dots, n_k} a_{n_1}\cdots a_{n_k} e^{-(\lambda_{n_1}+\cdots+\lambda_{n_k})\sigma}c_{\lambda_{n_1}+\cdots+\lambda_{n_k}}\textrm{is convergent}.$$
 By the Cauchy-Schwarz inequality, since $(a_n)$ is square summable, it is sufficient to prove that 
 $$\sum_{n_1,\dots, n_k} e^{-2\lambda_{n_1}\sigma}\cdots e^{-2\lambda_{n_k}\sigma} |c_{\lambda_{n_1}+\cdots+\lambda_{n_k}}|^2<+\infty.$$
Rewriting this
$$\sum_l |c_{\mu_l}|^2 e^{-2\mu_l\sigma} \textrm{card}\{(n_1,\dots,n_k):\lambda_{n_1}+\cdots+\lambda_{n_k}=\mu_l\}<+\infty$$
this follows from the assumption. 
\end{proof}

\begin{corollary}
Let $\lambda$ be a frequency, $k\geq 1$ and $\sigma>(k-1)L(\lambda)/2$. Then $T_\sigma$ maps $\mathcal H_2(\lambda)$ into $\mathcal H_{2k}(\lambda)$.
\end{corollary}
\begin{proof}
Let $\veps>0$. There exists $N\geq 1$ such that, for all $n\geq N$, $\log(n)/\lambda_n\leq L(\lambda)+\veps$.
Let $\mu\in (\lambda^*)^k$ and $n_1,\dots,n_k$ be such that $\lambda_{n_1}+\cdots+\lambda_{n_k}=\mu$.
Then each $\lambda_{n_i}$ is smaller than $\mu$ so that either $n_i\leq N$ or $n_i\leq \exp(\mu(L(\lambda)+\veps))$. 
Since the knowledge of $n_1,\dots,n_{k-1}$ determines the value of $n_k$, we have
$$\card\left\{(n_1,\dots,n_k):\ \lambda_{n_1}+\cdots+\lambda_{n_k}=\mu\right\}\leq N^{k-1}\exp\big((k-1)\mu(L(\lambda)+\veps)\big).$$
Taking the logarithm and letting $\mu$ to $+\infty$, we find $A(\lambda,k)\leq \frac{(k-1)(L(\lambda)+\veps)}{2}$, 
hence the inequality $A(\lambda,k)\leq \frac{(k-1)L(\lambda)}{2}$ since $\veps$ is arbitrary.
\end{proof}

\begin{corollary}
Let $\lambda$ be a frequency such that $L(\lambda)=0$. Then $T_\sigma$ maps boundedly $\mathcal H_2(\lambda)$ into $\mathcal H_q(\lambda)$ for all $q\geq 2$. 
\end{corollary}

\begin{question}
Let $p\geq 2$ and let $\lambda$ be a  frequency. Does $T_\sigma$ maps $\mathcal H_2(\lambda)$ into $\mathcal H_p(\lambda)$ as soon as $\sigma>\frac{(p-2)L(\lambda)}{4}$?
\end{question}

\begin{example}
Let $\lambda=(\log n)$. Then for all $k\geq 1$, $A(\lambda,k)=0$.
\end{example}
\begin{proof}
We first observe that $(\lambda^*)^k=\lambda$. Pick now $\log n\in\lambda$. We want to know 
the cardinal number of $\{(n_1,\dots,n_k)\in\NN:\ n_1\times\cdots\times n_k=n\}$. Decompose $n$ into a product
of prime numbers, $n=p_1^{\alpha_1}\cdots p_r^{\alpha_r}$. Then each $n_k$ writes $p_1^{\alpha_1(k)}\cdots p_r^{\alpha_r(k)}$ with $\alpha_j(1)+\cdots+\alpha_j(r)=\alpha_j$, $1\leq j\leq r$. Hence, $(\alpha_j(1),\cdots\alpha_j(r))$
is a weak composition of $\alpha_j$ into $k$ parts which can be done in $\binom{\alpha_i+k-1}{k-1}$ ways. In total, there are 
$$\prod_{i=1}^r \binom{\alpha_i+k-1}{k-1}\leq \prod_{i=1}^r (\alpha_i+k)^k$$
ways to write $n$ as a product of $k$ factors. Thus, 
$$A(\lambda,k)\leq\limsup_{n=\prod_{i=1}^r p_i^{\alpha_i}\to+\infty}\frac{\sum_{i=1}^r k\log(\alpha_i+k)}{2\sum_{i=1}^r \alpha_i\log(p_i)}=0.$$
\end{proof}

\smallskip

We finish this section by exhibiting a frequency $\lambda$ satisfying (BC) and such that, for all $k\geq 1$, $T_\sigma$ maps
$\mathcal H_2(\lambda)$ into $\mathcal H_{2k}(\lambda)$ if and only if $\sigma\geq A(\lambda,k)=\frac{k-1}{2k}$. We begin with two combinatorial lemmas.

\begin{lemma}\label{lem:combi1}
Let $b,c>0$, let $n\in\NN$ and let $\lambda_j=b+jc$, $j\geq 0$. For all $k\in \NN$, there exist $\gamma_k\in (0,1]$ and $\delta_k>0$ such that, for all $n\geq 2^{k}$, for all $\ell\in [(k-\gamma_k)n,(k+\gamma_k)n]\cap\NN_0$, 
$$\card\left\{(j_1,\dots,j_k)\in\{0,\dots,2n\}^k:\ \lambda_{j_1}+\cdots+\lambda_{j_k}= kb+\ell c\right\}\geq \delta_k n^{k-1}.$$
\end{lemma}
\begin{proof}
We define the sequences $(\gamma_k)$ and $(\delta_k)$ by $\gamma_k=2^{-(k-1)}$ and $\delta_1=1$, $\delta_{k+1}=\delta_k\cdot \gamma_{k+1}$. We proceed by induction over $k$. The case $k=1$ is trivial. Assume that the result has been proved for $k$ and let us prove it for $k+1$. Let $n\geq 2^{k+1}$. Choose $j_{k+1}$ any integer in 
$[(1-\gamma_{k+1})n,(1+\gamma_{k+1})n]$ and $\ell\in [(k+1-\gamma_{k+1})n,(k+1+\gamma_{k+1})n]\cap\NN_0$. Then
\begin{equation}\label{eq:lemcombi1}
\lambda_{j_1}+\cdots+\lambda_{j_{k+1}}=(k+1)b+\ell c \iff \lambda_{j_1}+\cdots+\lambda_{j_k}=kb+(\ell-{j_{k+1}})c.
\end{equation}
Now,
$$\left|\ell-{j_{k+1}}-kn\right|\leq 2\gamma_{k+1}n=\gamma_kn $$
so that there exist at least $\delta_k n^{k-1}$ choices of $(j_1,\dots,j_k)$ such that \eqref{eq:lemcombi1} is true, $j_{k+1}$ being fixed. Now, there are $2\lfloor \gamma_{k+1}n\rfloor+1$ choices of $j_{k+1}$ and since $\gamma_{k+1}n-1\geq \gamma_{k+1}n/2$ because $\gamma_{k+1}n\geq 2$, we get the result.
\end{proof}

\begin{lemma}\label{lem:combi2}
Let $(b_n)$ and $(c_n)$ be two sequences of positive real numbers such that the sequences $(b_1,\dots,b_N,c_1,\dots,c_N)$ are $\ZZ$-independent for all $N\geq 1$, $2n+1\leq\exp(b_n)$ and $nc_n\leq 1$ for each $n\in\NN$. Define a sequence $(\lambda_n)$ by $\lambda_{m^2+j}=b_m+jc_m$, $m\geq 1$, $j=0,\dots,2m$. Then for all $k>0$ there exists $C_k>0$ such that, for all $\mu>0$, 
$$\card\left\{(n_1,\dots,n_k)\in\NN^k:\ \lambda_{n_1}+\cdots+\lambda_{n_k}=\mu\right\}\leq
C_k\exp\left(\frac{(k-1)\mu}k\right).$$
\end{lemma}
\begin{proof}
If $\mu$ can be written $\mu=\lambda_{n_1}+\cdots+\lambda_{n_k}$ for some sequence $(n_1,\dots,n_k)$, it can be uniquely written 
\begin{equation}\label{eq:lemcombi2bis}
\mu=\alpha_1 b_{r_1}+\beta_1 c_{r_1}+\cdots+\alpha_l b_{r_l}+\beta_l c_{r_l}
\end{equation}
with $1\leq l\leq k$, $r_1<r_2<\cdots<r_l$, $\alpha_i\geq 1$, $\alpha_1+\cdots+\alpha_l=k$ and $0\leq \beta_i\leq 2\alpha_i r_i$. We will first estimate $\card\left\{(n_1,\dots,n_k)\in\NN^k:\ \lambda_{n_1}+\cdots+\lambda_{n_k}=\mu\right\}$ by a quantity depending on $k$, $l$, $\alpha_i$, $r_i$ and $\beta_i$. In view of the definition of the sequence $\lambda$ and of \eqref{eq:lemcombi2bis}, we are reduced to estimate the number of $2k$-tuples $(m_1,\dots,m_k,j_1,\dots,j_k)$ such that
 for all $s=1,\dots,k$, $0\leq j_s\leq 2m_s$ and, for all $i=1,\dots,l$,
\begin{itemize}
\item there are $\alpha_i$ elements in $m_1,\dots,m_k$ which are equal to $r_i$;
\item if $\phi_i(1),\dots,\phi_i(\alpha_i)$ are the indices of these elements, then 
\begin{equation}\label{eq:lemcombi2}
j_{\phi_i(1)}+\cdots+j_{\phi_i(\alpha_i)}=\beta_i.
\end{equation}
\end{itemize}
We first choose the values of $m_1,\dots,m_k$. We choose the $\alpha_1$ indices in $\{1,\dots,k\}$ such that the corresponding $m_i$ are equal to $r_1$. We then do the same for the $\alpha_2$ elements equal to $r_2$ and so on until $k-1$ (the remaining $m_i$ are fixed and equal to $r_l$). Thus the number of choices for $m_1,\dots,m_k$ is equal to 
$$\binom{k}{\alpha_1}\times\binom{k-\alpha_1}{\alpha_2}\times\cdots\times\binom{k-(\alpha_1+\cdots+\alpha_{l-2})}{\alpha_{l-1}}.$$
Because $l\leq k$ and $\alpha_1+\cdots+\alpha_l=k$, this number can be bounded from above by some number depending only on $k$. The integers $m_1,\dots,m_k$ having been fixed, we now choose the integers $j_1,\dots,j_k$. For each $i\in\{1,\dots,l\}$, \eqref{eq:lemcombi2} implies that there are at most $(2r_i+1)^{\alpha_i-1}$ choices for the values of 
$j_{\phi_i(1)},\dots,j_{\phi_i(\alpha_i)}$: indeed, each $j_{\phi_i(t)}$ belongs to $\{0,\dots,2r_i\}$ and the last one is fixed when we know the values of the first $\alpha_i-1$ ones). Finally, we have found that
\begin{align*}
\card\left\{(n_1,\dots,n_k)\in\NN^k:\ \mu=\lambda_{n_1}+\cdots+\lambda_{n_k}\right\}& \leq C_k\prod_{i=1}^l (2r_i+1)^{\alpha_i-1}\\
&\leq C_k \exp\left(\sum_{i=1}^l (\alpha_i-1)b_{r_i}\right)
\end{align*}
whre the last inequality follows from the assumption $2n+1\leq\exp(b_n)$ for all $n\in\mathbb N$. 
Now we have $\sum_{i=1}^l \alpha_i b_{r_i}\leq\mu$ and 
\begin{align*}
k\sum_{i=1}^l b_{r_i}&\geq \sum_{i=1}^l \alpha_i b_{r_i}= \mu-\sum_{i=1}^l \beta_i c_{r_i}\\
&\geq \mu-2\sum_{i=1}^l \alpha_i r_i c_{r_i} \geq \mu-2k\sum_{i=1}^l r_ic_{r_i}\\
&\geq \mu-2k^2
\end{align*}
since $nc_n\leq 1$ and $l\leq k$. This implies that 
$$\sum_{i=1}^l (\alpha_i-1)b_{r_i}\leq \mu-\frac\mu k +2k=\frac{(k-1)\mu}{k}+2k,$$
hence the result.
\end{proof}

\begin{theorem}
There exists a frequency $(\lambda_n)$ satisfying (BC) such that for all $k\geq1$, $T_\sigma$ maps $\mathcal H_2(\lambda)$ into $\mathcal H_{2k}(\lambda)$ if and only if $\sigma\geq\frac{k-1}{2k}=A(\lambda,k)$.
\end{theorem}
\begin{proof}
Let $(b_n)$ and $(c_n)$ be two sequences of positive real numbers such that
\begin{itemize}
\item for all $n\geq 1$, $\log(2n+1)\leq b_n\leq\log(2n+2)$;
\item for all $n\geq 1$, $(b_{n+1}-b_n)/8n\leq c_n\leq (b_{n+1}-b_n)/4n$;
\item for all $N\geq 1$, the sequences $(b_1,\dots,b_N,c_1,\dots,c_N)$ are $\ZZ$-independent.
\end{itemize}
We then define $\lambda$ by $\lambda_{m^2+j}=b_m+jc_m$, $m\geq 1$, $j=0,\dots,2m$.  We may argue as in the proof of Theorem \ref{thm:H1H2} to show that the frequency $\lambda$ satisfies (BC). Using Proposition \ref{prop:H2H2K} (look at the first sentence of the proof) and Lemma \ref{lem:combi2}, we get easily that $T_\sigma$ maps $\mathcal H_2(\lambda)$ into $\mathcal H_{2k}(\lambda)$ for $\sigma\geq\frac{k-1}{2k}$ and also that $A(\lambda,k)\leq\frac{k-1}{2k}$.

Conversely, assume that $T_\sigma$ maps $\mathcal H_2(\lambda)$ into $\mathcal H_{2k}(\lambda)$ (boundedness is automatic by the closed graph theorem). Let us consider $D^{[m]}=\sum_{j=0}^{2m}e^{-\lambda_{m^2+j}s}$ for $m\geq 1$, so that $\|D^{[m]}\|_2=\sqrt{2m+1}$. Write $\lambda^{*k}$ as the increasing sequence $(\mu_l)$ and observe that
$$(T_\sigma D^{[m]})^k=\sum_l a_{\mu_l}e^{-\mu_l\sigma}e^{-\mu_l s}$$
where $a_{\mu_l}=\card\left\{(j_1,\dots,j_k)\in\{0,\dots,2m\}^k:\ \mu_l=kb_m+(j_1+\cdots+j_k)c_m\right\}$. Lemma \ref{lem:combi1} tells us that, for $m$ sufficiently large, there is at least $\gamma_k m$ terms of the sequence $(\mu_l)$ so that $a_{\mu_l}\geq \delta_k m^{k-1}$. Furthermore, for those $\mu_l$,
$$\mu_l\leq k(b_m+2mc_m)\leq k\log m+A_k.$$
In particular, 
$$\frac{\log a_{\mu_l}}{2\mu_l}\geq \frac{(k-1)\log m+\log \delta_k}{2(k\log m+A_k)}$$
which shows that $A(\lambda,k)\geq\frac{k-1}{2k}$. Furthermore,
\begin{align*}
\|T_\sigma D^{[m]}\|_{2k}=\|(T_\sigma D^{[m]})^k\|_2^{1/k}&\geq C_k\left(m\cdot m^{2(k-1)}\cdot m^{-2\sigma k}\right)^{1/2k}\\
&\geq C_k m^{\frac1{2k}+\frac{k-1}k-\sigma}.
\end{align*}
Therefore, the boundedness of $T_\sigma$ from $\mathcal H_2(\lambda)$ into $\mathcal H_{2k}(\lambda)$ implies that 
$$ m^{\frac 1{2k}+\frac{k-1}{k}-\sigma}\leq C'_k \sqrt{2m+1}$$
for all sufficiently large $m$, which itself yields
$\sigma\geq\frac{k-1}{2k}$.
\end{proof}

\begin{question}
Let $p\geq 2$. Is it true that, for the previous sequence $(\lambda_n)$, $T_\sigma$ maps $\mathcal H_2(\lambda)$ into $\mathcal H_p(\lambda)$ if and only if $\sigma\geq\frac{p-2}{2p}$?
\end{question}


\section{Other results}

\subsection{Norm of the projection in $\mathcal H_1(\lambda)$}

In \cite{DS19}, Defant and Schoolmann have shown, using a vector-valued argument, that for all frequencies $\lambda$ and for all $N\geq 1$, $\|S_N\|_{\mathcal H_1(\lambda)\to\mathcal H_1(\lambda)}\leq  \|S_N\|_{\mathcal H_\infty(\lambda)\to\mathcal H_\infty(\lambda)}$. We provide a different approach to estimate $\|S_N\|_{\mathcal H_1(\lambda)\to\mathcal H_1(\lambda)}$, inspired by \cite{BBSS19}.
\begin{proposition}\label{prop:projection}
Let $\lambda$ be a frequency. The, for all $N\geq 1$, $\|S_N\|_{\mathcal H_1\to\mathcal H_1}\leq C\log(\Lambda_N)$ where 
$$\Lambda_N=\sup \left\{\frac{\left\|\sum_{n=1}^N a_n e^{-\lambda_n s}\right\|_2}
{\left\|\sum_{n=1}^N a_n e^{-\lambda_n s}\right\|_1}:\ a_1,\dots,a_N\in\CC\right\}.$$
\end{proposition}

\begin{proof}
We work in $H_1^\lambda(G)$ where $(G,\beta)$ is a $\lambda$-Dirichlet group. Let $g\in L_1(G)$. Then, for $\veps\in(0,1)$, 
\begin{align*}
\int |g|&=\int |g|^{\frac{2\veps}{1+\veps}}|g|^{\frac{1-\veps}{1+\veps}}\\
&\leq \left(\int |g|^2\right)^{\frac\veps{1+\veps}}
\left(\int |g|^{1-\veps}\right)^{\frac 1{1+\veps}} 
\end{align*}
where we have applied H\"older's inequality for $(1+\veps)/\veps$ and $1+\veps$. Applying this to $S_N f$, where $f\in H_1^\lambda(G)$, we get
$$\|S_N f\|_1\leq \|S_N f\|_2^{\frac{2\veps}{1+\veps}}\|S_N f\|_{1-\veps}^\frac{1-\veps}{1+\veps}.$$
We now use a result of Helson \cite{Hel58}, saying that
$$\|S_N f\|_{1-\veps}\leq \frac A\veps \|f\|_1$$
where the constant $A$ is absolute. Therefore, assuming $\|f\|_1=1$, after some simplifications, we get
\begin{align*}
\|S_N f\|_1&\leq \left(\frac A\veps\right)\exp\left(\frac{2\veps}{1-\veps}\log(\Lambda_N)\right).
\end{align*}
We conclude by choosing $\veps=1/\log(\Lambda_N)$.
\end{proof}
\begin{corollary}
There exists $C>0$ such that, for all frequency $\lambda$, $\|S_N\|_{\mathcal H_1(\lambda)\to\mathcal H_1(\lambda)}\leq C\log(N)$.
\end{corollary}
\begin{proof}
We get immediately that $\Lambda_N\leq \sqrt N$ by using the Cauchy-Schwarz inequality and the fact that $|a_n|\leq \|D\|_1$ for all $D=\sum_{n=1}^N a_n e^{-\lambda_n s}$.
\end{proof}

This last corollary has an interest provided we are unable to prove that $\|S_N\|_{\mathcal H_\infty(\lambda)\to\mathcal H_\infty(\lambda)}$ is less than $C\log(N)$. The best known estimation on $\|S_N\|_{\mathcal H_\infty(\lambda)\to\mathcal H_\infty(\lambda)}$ is given by Theorem \ref{thm:BohrNCHinf} and indeed it provides worst estimations for some sequences $\lambda$. Indeed, pick the sequence $\lambda$ defined in Example \ref{ex:sequenceNC}. Let $N=2^n$ for some $n$ and pick $M>N$. Then, if $M<2^{n+1}$, then 
$$\log\left(\frac{\lambda_M+\lambda_N}{\lambda_M-\lambda_N}\right)\geq \frac 12e^{n^2}$$
whereas, if $M\geq 2^{n+1}$, then 
$$M-N-1\geq 2^n -1.$$
A variant of Proposition \ref{prop:projection} was already used in the classical case $\lambda=(\log n)$ to prove that $\|S_N\|_{\mathcal H_1\to\mathcal H_1}\leq C\frac{\log N}{\log\log N}$. A precise solution to the problem on how large can be $\Lambda_N$ in this case can be found in \cite{DP17}.


\providecommand{\bysame}{\leavevmode\hbox to3em{\hrulefill}\thinspace}
\providecommand{\MR}{\relax\ifhmode\unskip\space\fi MR }
\providecommand{\MRhref}[2]{%
  \href{http://www.ams.org/mathscinet-getitem?mr=#1}{#2}
}
\providecommand{\href}[2]{#2}

\end{document}